\documentclass[a4paper, 11pt]{article}
\usepackage[margin=25mm]{geometry}
\usepackage{amsmath}
\usepackage{amssymb}
\usepackage{amsthm}
\usepackage{bbm}
\usepackage{titlesec}
\usepackage{titlefoot}
\usepackage{comment} 
\usepackage{listings}
\usepackage{mathtools}
\usepackage{hyperref}
\usepackage{comment}
\titleformat*{\section}{\normalsize\bfseries}
\allowdisplaybreaks
\def\R{\mathbb{R}}

\def\p{\partial}
\theoremstyle{plain}
\newtheorem{thm}{Theorem}[section]
\newtheorem{prop}[thm]{Proposition}

\newtheorem{lemma}[thm]{Lemma}
\newtheorem{rem}[thm]{Remark}

\numberwithin{equation}{section}

\makeatletter
\@addtoreset{equation}{section}

\makeatother

\begin{document}

\title{
\vspace{-1cm}
\large{\bf Higher-order Asymptotic Profiles of Solutions to the Cauchy Problem for the Convection-Diffusion Equation with Variable Diffusion}}
\author{Ikki Fukuda and Shinya Sato}
\date{}
\maketitle

\vspace{-0.8cm}
\begin{center}
Dedicated to Professor Takayoshi Ogawa on the occasion of his 60th birthday
\end{center}


\vspace{-0.8cm}
\begin{abstract}
We consider the asymptotic behavior of solutions to the convection-diffusion equation: 
\[
\partial_t u -  \mathrm{div}\left(a(x)\nabla u\right) = d\cdot\nabla  \left(\left\lvert u\right\rvert ^{q-1}u\right),\ \  x\in\mathbb{R}^n, \ t>0
\]
with an integrable initial data $u_{0}(x)$, where $n\ge1$, $q>1+\frac{1}{n}$ and $d\in \R^{n}$. Moreover, we take $a(x)=1+b(x)>0$, where $b(x)$ is smooth and decays fast enough at spatial infinity. It is known that the asymptotic profile of the solution to this problem can be given by the heat kernel. Moreover, some higher-order asymptotic expansions of the solution have already been studied. In particular, the structures of the second asymptotic profiles strongly depend on the nonlinear exponent $q$. More precisely, these profiles have different decay orders in each of the following three cases: $1+\frac{1}{n}<q<1+\frac{2}{n}$; $q=1+\frac{2}{n}$; $q>1+\frac{2}{n}$. In this paper, we focus on the critical case $q=1+\frac{2}{n}$. By analyzing the corresponding integral equation in details, we have succeeded to give the more higher-order asymptotic expansion of the solution, which generalizes the previous works.  
\end{abstract}


\section{Introduction}

We consider the Cauchy problem for the convection-diffusion equation of the following form:
\begin{equation}\label{CD}
  \begin{aligned}
    &\partial_t u -  \mathrm{div}\left(a(x)\nabla u\right) = d\cdot\nabla  \left(\left\lvert u\right\rvert ^{q-1}u\right),\ \  x\in\mathbb{R}^n, \ t>0,\\
    &u(x,0)=u_0(x),\ \  x\in\mathbb{R}^n, 
\end{aligned}
\end{equation}
where $u=u(x, t)$ is a real-valued unknown function, $u_{0}(x)$ is a given initial data, $n\ge1$, $q>1$ and $d\in \R^{n}$. 
For the variable diffusion coefficient $a(x)$, we take asymptotically constant diffusion of the form $a(x)=1+b(x)>0$ with $b(x)$ satisfying $b\in L^{1}(\R^{n})\cap C^{1, \alpha}(\R^{n})$ ($0<\alpha <1$) and $\|b\|_{L^{\infty}}<1$ such that 
\begin{equation}\label{C-b}
|b(x)|+\left(1+|x|^{2}\right)^{\frac{1}{2}}\left|\nabla b(x)\right| \le C\left(1+|x|^{2}\right)^{-\frac{\delta}{2}}, \ \ x\in \R^{n}, 
\end{equation}
for some positive constants $C>0$ and $\delta>0$. 
More generally, our study is valid as long as the semigroup $T(t)$ generated by the corresponding linear diffusion equation 
\[
\p_{t}u- \mathrm{div}\left(a(x)\nabla u\right)=0, \ \  x\in\mathbb{R}^n, \ t>0
\] 
satisfies the following estimate: 
\[
\left\|\nabla T(t)(u_{0})\right\|_{L^{p}}\le C\left\|u_{0}\right\|_{L^{r}}t^{-\frac{n}{2}\left(\frac{1}{r}-\frac{1}{p}\right)-\frac{1}{2}}, \ \ t>0, \ 1\le r\le p\le \infty
\]
with some positive constant $C=C(p, r)>0$ independent of $a(x)$. Actually, if $b(x)$ satisfies \eqref{C-b} for some $C>0$ and $\delta>0$, then the above estimate holds (for details, see Murata~\cite{M85}). For the more classical theory of parabolic operators, we can also refer to \cite{A68-1, A68-2, F64}. 

This kind of equation arises from the study in the context of the diffusion of pollution in fluids (cf.~\cite{DDG96}). In this paper, we would like to analyze the asymptotic behavior of the solutions to \eqref{CD}. 
In particular, we are interested in how the variable diffusion and the nonlinearity affect the solution as $t\to \infty$.

First of all, let us recall some known results related to the Cauchy problem \eqref{CD}. 
In \cite{E-Z}, Escobedo--Zuazua studied \eqref{CD} when $a(x)\equiv1$ (i.e., $b(x)\equiv0$). 
They showed that if $q>1$ and $u_{0}\in L^{1}(\R^{n})$, then there exists a unique global solution $u \in C([0, \infty); L^1(\mathbb{R}^n))$ satisfying 
\begin{equation*}
      u \in C((0, \infty); W^{2, p}(\mathbb{R}^n))  \cap C^1((0, \infty); L^p(\mathbb{R}^n)), 
 \end{equation*}
 for any $p\in \left(1, \infty\right)$. Moreover, if we additionally assume $u_{0}\in L^{1}(\R^{n}) \cap L^\infty(\mathbb{R}^n)$, then the solution $u(x, t)$ also satisfies $u \in L^{\infty}(\mathbb{R}^n\times (0,\infty))$. 
 For the classical regularity theory of parabolic equations, let us refer to \cite{L68}. 
 Moreover, we note that the above results are also summarized very well, in the lecture note written by Zuazua \cite{Z1}. 

Next, let us introduce the additional properties of the solution $u(x, t)$ to \eqref{CD}. Now, integrating the equation \eqref{CD} over $\R^{n}$, we deduce that the mass is conserved, 
\[
\text{i.e.,} \ \ \int_{\R^{n}}u(x, t)dx=\int_{\R^{n}}u_{0}(x)dx, \ \ t>0. 
\]
On the other hand, the following $L^{p}$-decay estimates hold for any $p\in \left[1,\infty\right] $ (cf.~\cite{E-Z, Z1}): 
    \begin{align}
     & \left\lVert u(t) \right\rVert  _{L^p} \leq Ct^{-\frac{n}{2}\left(1-\frac{1}{p}\right)},\ \  t>0, \label{CDdecay} \\
      &\left\lVert \nabla u(t) \right\rVert  _{L^p} \leq Ct^{-\frac{n}{2}\left(1-\frac{1}{p}\right)-\frac{1}{2}},\ \  t>1. \label{CDdecay-2}
    \end{align}
Furthermore, Duro--Zuazua~\cite{DZ99} studied the more general cases of \eqref{CD} with the variable diffusion satisfying \eqref{C-b} and generalized the results given in \cite{E-Z}. Actually, the all properties mentioned above also can be established for our target problem \eqref{CD}. 
We note that the constants appear in \eqref{CDdecay} and \eqref{CDdecay-2} do not depend on $b(x)$, by virtue of \eqref{C-b}. 
Also, we emphasize that these decay rates are the same as for the solution to the corresponding linear problem. 

In addition to the decay estimates \eqref{CDdecay} and \eqref{CDdecay-2}, the more detailed asymptotic behavior of the solution has been studied in \cite{DZ99, E-Z}. 
The decay rates given in \eqref{CDdecay} and \eqref{CDdecay-2} do not depend on the nonlinear exponent $q$. 
However, the asymptotic behavior of the solution deeply depends on the nonlinearity. Actually, when $a(x)\equiv1$, the existence of a critical value of the exponent $q$. Roughly speaking, the following results have been shown in \cite{E-Z}:
\begin{itemize}

\item
If $q>1+\frac{1}{n}$, the solution behaves like a self-similar solution of the linear heat equation. 

\item
If $q=1+\frac{1}{n}$, the solution behaves like a self-similar solution of the full equation. 

\item
If $1<q<1+\frac{1}{n}$, the solution behaves like a self-similar solution of a reduced equation. 

\end{itemize} 
For the more rigorous statements, see \cite{E-Z}. In the present paper, we focus on the super critical case $q>1+\frac{1}{n}$. 
Under this situation, the asymptotic profile of the solution can be given by the heat kernel, even if the diffusion coefficient $a(x)$ is the more general case (cf.~\cite{DZ99}). 
That is, neither the effects of the variable diffusion, nor the nonlinearity appear in the first approximation. In order to state such a result, we set 
\begin{equation}\label{heat}
G(x, t) \coloneqq \frac{1}{(4\pi t)^{\frac{n}{2}}}e^{-\frac{|x|^{2}}{4t}}, \ \ M:=\int_{\R^{n}}u_{0}(x)dx. 
\end{equation}
Moreover, for $p\in [1, \infty)$, let us introduce the following weighted Lebesgue spaces:  
\[
L^{p}( \mathbb{R}^n; 1+|x| ):=\left\{f\in L^{p}(\R^{n}); \ \int_{\R^{n}}\left\{|f(x)|(1+|x|)\right\}^{p}dx<\infty\right\}. 
\]
Then, the following asymptotic formula has been established in \cite{DZ99}: 
\begin{prop}[\cite{DZ99}]\label{CDasympthm}
Let $n\ge1$ and $ q>1+\frac{1}{n} $. Suppose $ u_0 \in L^{1}( \mathbb{R}^n; 1+|x| )\cap L^{\infty }( \mathbb{R}^n )$. 
Then, for any $ p \in [1, \infty ] $ and $t\ge1$, the solution $ u(x,t) $ to \eqref{CD} satisfies the following asymptotic formula:
    \begin{equation}\label{CDasymp}
      \left\lVert u(t) - MG(t)\right\rVert  _{L^p} \le C
      \begin{cases}
        t^{-\frac{n}{2}\left(q-\frac{1}{p}\right)+\frac{1}{2}}, &\displaystyle 1 + \frac{1}{n} < q < 1 + \frac{2}{n}, \\[3mm]
        t^{-\frac{n}{2}\left(1-\frac{1}{p}\right) - \frac{1}{2}}\log (2+t), &\displaystyle q=1+\frac{2}{n},\\[3mm]
        t^{-\frac{n}{2}\left(1-\frac{1}{p}\right) - \frac{1}{2}}, &\displaystyle q > 1 + \frac{2}{n},
    \end{cases}
    \end{equation}
    where $G(x, t)$ and $M$ are defined by \eqref{heat}. 
    \end{prop}
As we can see from the above result, the asymptotic rates given in \eqref{CDasymp} deeply depend on the exponent $q$. 
In addition, the following three cases are distinguished: 
\[1+\frac{1}{n}<q<1+\frac{2}{n}; \ \ q=1+\frac{2}{n}; \ \ q>1+\frac{2}{n}.\] 
Roughly speaking, we note that the above critical exponent $q=1+\frac{2}{n}$ appears due to the integrability of the following integral:
\begin{equation}\label{int}
  \int_{1}^{t} \int_{\mathbb{R}^n}  \left\lvert u(x,\tau)\right\rvert^q \,dx  \,d\tau = \int_{1}^{t} \left\lVert u(\tau)\right\rVert _{L^q}^{q} \,d\tau \leq C\int_{1}^{t} \tau^{-\frac{n}{2}\left(q-1\right)} \,d\tau, 
\end{equation}
where we used the decay estimate \eqref{CDdecay}. Based on the above result, we can expect that the nonlinearity strongly affects the second asymptotic profiles of the solution. Actually, such results have already been obtained in \cite{D-C, Z2}. First, Zuazua~\cite{Z2} succeeded to derive the three kinds of second asymptotic profiles of the solution $u(x, t)$ depending on $q$, when $a(x)\equiv 1$. After that, these results were generalized by Duro--Carpio~\cite{D-C}, including the variable diffusion cases. From such results, in view of the second asymptotic profiles, we can see that the asymptotic rates to $MG(x, t)$ given in \eqref{CDasymp} are optimal, with respect to the time decaying order. On the other hand, the effect of the perturbation $b(x)$ becomes large in lower dimensions, especially in the one-dimensional case $n=1$. 
Actually, in \cite{D-C}, they proved that the variable diffusion appears in the second asymptotic profiles only when $q>1+\frac{2}{n}$ or when $n=1$ and $q=3$. 
More precisely, the following three propositions have been shown in \cite{D-C}: 
\begin{prop}[\cite{D-C}]\label{prop.DC-sub}
Let $n\ge1$ and $ 1+\frac{1}{n}<q<1+\frac{2}{n} $. Suppose $ u_0 \in L^{1}( \mathbb{R}^n; 1+|x| )\cap L^{n(q-1)}( \mathbb{R}^n)$. 
Then, for any $ p \in [1, \infty ] $, the solution $ u(x, t) $ to \eqref{CD} satisfies the following asymptotic formula:
\begin{equation}\label{DC-sub}
        \lim_{t\to \infty}t^{\frac{n}{2}\left(q-\frac{1}{p}\right)-\frac{1}{2}}\left\lVert u(t)-MG(t)-\left|M\right|^{q-1}MZ(t)\right\rVert _{L^p}=0, 
\end{equation}
where $G(x, t)$ and $M$ are defined by \eqref{heat}, while $Z(x, t)$ is the solution to the following problem: 
\begin{align*}
&\p_{t}Z-\Delta Z=d\cdot \nabla\left(G^{q}\right), \ \ x\in \R^{n}, \ t>0, \\
&Z(x, 0)=0, \ \ x\in \R^{n}. 
\end{align*}
    \end{prop}

\begin{prop}[\cite{D-C}]\label{prop.DC-critical}
Let $q=1+\frac{2}{n} $. Then, for any $ p \in [1, \infty ] $, the solution $ u(x, t) $ to \eqref{CD} satisfies the following asymptotic formulas:

\smallskip
\noindent
{\rm (i)} If $n\ge2$ and $ u_0 \in L^{1}( \mathbb{R}^n; 1+|x| )\cap L^{q}( \mathbb{R}^n; 1+|x|)\cap L^{2}( \mathbb{R}^n )$, then 
    \begin{equation}\label{Z-2nd}
        \lim_{t\to \infty}\frac{t^{\frac{n}{2}\left(1-\frac{1}{p}\right)+\frac{1}{2}}}{\log t}\left\lVert u(t)-MG(t)-\alpha_n\left\lvert M\right\rvert ^{\frac{2}{n}}M(\log t)d\cdot\nabla G(t)\right\rVert _{L^p}=0.
    \end{equation}
    \noindent
{\rm (ii)} If $n=1$ and $ u_0 \in L^{1}( \mathbb{R}; 1+|x| )\cap L^{3}( \mathbb{R}; 1+|x|)$, then 
    \begin{equation}\label{Z-2nd-2}
        \lim_{t\to \infty}\frac{t^{\frac{1}{2}\left(1-\frac{1}{p}\right)+\frac{1}{2}}}{\log t}\left\lVert u(t)-MG(t)-\left(K(t)+\alpha_1dM^{3}\right)(\log t)\p_{x}G(t)\right\rVert _{L^p}=0. 
    \end{equation}
    Here, $G(x, t)$ and $M$ are defined by \eqref{heat}, while $\alpha_{n}$ and $K(t)$ are defined by 
    \begin{equation}\label{alpha_n}
    \alpha_{n}:=\frac{1}{4\pi}\left(1+\frac{2}{n}\right)^{-\frac{n}{2}}, \ \ K(t):=\frac{1}{\log t}\int_{0}^{t}\int_{\R}b(y)\p_{x}u(y, \tau)dyd\tau. 
    \end{equation}
    \end{prop}
\begin{rem}
In the above proposition, we must distinguish the cases $n\ge2$ and $n=1$, due to the integrability for the integral $\int_{1}^{\infty}\int_{\R^{n}}\left|b(y)\nabla u(y, \tau)\right|dyd\tau$. For details, see Remark~\ref{remark1} below. However, in the case of $n=1$, if we further assume that the initial data can be expressed as $u_{0}(x)=\p_{x}v_{0}(x)$ with some $v_{0}\in L^{1}(\R)$, then we have the following fact (for the proof, see {\rm \cite{D-C}}): 
\[
\lim_{t\to \infty}K(t)=\lim_{t\to \infty}\left\{\frac{1}{\log t}\int_{0}^{t}\int_{\R}b(y)\p_{x}u(y, \tau)dyd\tau\right\}=0. 
\]
Therefore, under this situation, we can see that \eqref{Z-2nd} holds even if the case of $n=1$. 
\end{rem}

\begin{prop}[\cite{D-C}]\label{prop.DC-super}
Let $q>1+\frac{2}{n} $. Then, for any $ p \in [1, \infty ] $, the solution $ u(x, t) $ to \eqref{CD} satisfies the following assertions:

\smallskip
\noindent
{\rm (i)} If $n\ge2$ and $ u_0 \in L^{1}( \mathbb{R}^n; 1+|x| )\cap L^{n(q-1)}( \mathbb{R}^n)$, then 
    \begin{equation}\label{DC-1}
        \lim_{t\to \infty}t^{\frac{n}{2}\left(1-\frac{1}{p}\right)+\frac{1}{2}}\left\lVert u(t)-MG(t)-\beta \cdot\nabla G(t)\right\rVert _{L^p}=0, 
    \end{equation}
    where $G(x, t)$ and $M$ are defined by \eqref{heat}, while $\beta$ is defined by 
    \begin{align}
    &\beta:=d\int_{0}^{\infty}\int_{\R^{n}}\left(\left|u\right|^{q-1}u\right)(y, \tau)dyd\tau+\int_{0}^{\infty}\int_{\R^{n}}b(y)\nabla u(y, \tau)dyd\tau-m, \label{beta}\\
    &  m\coloneqq\left(m_{1}, m_{2}, \cdots, m_{n}\right), \ \ m_i\coloneqq\int_{\mathbb{R}^n} x_i u_0(x)dx. \label{m}
    \end{align}
    
    \noindent
{\rm (ii)} If $n=1$, $ u_0 \in L^{1}( \mathbb{R}; 1+|x| )\cap L^{q}( \mathbb{R}; 1+|x|)$ and the initial data $u_{0}(x)$ can be expressed as $u_{0}(x)=\p_{x}v_{0}(x)$ for some $v_{0}\in L^{1}(\R)$, then \eqref{DC-1} holds with $\beta$ as in \eqref{beta}. 
       \end{prop}

\begin{rem}\label{rem-const-diffu}
We note that the term 
\[
\int_{0}^{\infty}\int_{\R^{n}}b(y)\nabla u(y, \tau)dyd\tau=-\int_{0}^{\infty}\int_{\R^{n}}\nabla b(y)u(y, \tau)dyd\tau\] 
does not appear if $b(x)$ is an identically constant. This fact agrees with the result given in {\rm \cite{Z2}}. 
\end{rem}

Comparing the asymptotic rates $o(t^{-(n/2)(q-1/p)+1/2})$ in \eqref{DC-sub}, $o( t^{-(n/2)(1-1/p)-1/2}\log t )$ in \eqref{Z-2nd} and \eqref{Z-2nd-2}, and $o( t^{-(n/2)(1-1/p)-1/2})$ in \eqref{DC-1} obtained from the three propositions above, we can find that $o( t^{-(n/2)(1-1/p)-1/2})$ in \eqref{DC-1} gives the fastest asymptotic rate. Noticing that this asymptotic rate is the same as the asymptotic rate which appears in the second order asymptotic expansion for the solution to the linear heat equation (see, \eqref{HEasymp2} below). 
Comparing this result with the other two cases, it seems that the asymptotic formulas in Propositions~\ref{prop.DC-sub} and \ref{prop.DC-critical} do not give a good approximation to the solution.
In addition, it should also be noted that a very interesting phenomenon occurs in the critical case $q=1+\frac{2}{n}$. 
More precisely, the logarithmic term appears not only in the asymptotic rate $o( t^{-(n/2)(1-1/p)-1/2}\log t )$ but also in the structures of the second asymptotic profiles: 
\begin{equation}\label{second}
\alpha_n\left\lvert M\right\rvert ^{\frac{2}{n}}M(\log t)d\cdot\nabla G(x, t) \ \ \text{in \eqref{Z-2nd}}, \ \ 
\left(K(t)+\alpha_1dM^{3}\right)(\log t)\p_{x}G(x, t) \ \ \text{in \eqref{Z-2nd-2}}. 
\end{equation}
Moreover, we remark that the logarithmic term comes from both of the effects of the nonlinearity and the variable diffusion. Also, we note that a similar phenomenon also can be observed for other nonlinear PDEs. 
For example, let us introduce the following parabolic system of chemotaxis: 
\begin{align*}
\begin{split}
&u_{t}=\Delta u-\nabla \cdot (u\nabla v), \ \ x\in \R^{n}, \ t>0, \\
&v_{t}=\Delta v-v+u, \ \ x\in \R^{n}, \ t>0. 
\end{split}
\end{align*}
Indeed, in Nagai--Yamada \cite{NY07}, the related result to \eqref{Z-2nd} and \eqref{Z-2nd-2} has been derived for the above system, in the case of $n=1$ (we can also refer to \cite{KM09}). 

On the other hand, Ishige--Kawakami~\cite{I-K} succeeded to derive the different asymptotic formula for any $q>1+\frac{1}{n}$. 
Actually, the following proposition has been established: 
\begin{prop}[\cite{I-K}]\label{IK}
Let $n\ge1$ and $q>1+\frac{1}{n}$. Suppose $\ u_0\in L^1( \mathbb{R}^n;1+\left\lvert x\right\rvert )\cap L^\infty(\mathbb{R}^n) $ and $a(x)\equiv1$. Then, for any $ p \in [1, \infty ] $, the solution $ u(x,t) $ to \eqref{CD} satisfies the following asymptotic formula:
 \begin{equation}\label{IK-ap}
\left\|u(t)-\hat{u}(t)\right\|_{L^{p}}=\begin{cases}
        O\left(t^{-\frac{n}{2}\left(1-\frac{1}{p}\right)-2\gamma}\right), &\displaystyle 2\gamma<\frac{1}{2},\\[4mm]
        O\left(t^{-\frac{n}{2}\left(1-\frac{1}{p}\right)-\frac{1}{2}}\log t\right), &\displaystyle 2\gamma=\frac{1}{2},\\[4mm]
        o\left(t^{-\frac{n}{2}\left(1-\frac{1}{p}\right)-\frac{1}{2}}\right), &\displaystyle 2\gamma>\frac{1}{2}
    \end{cases}  
 \end{equation}
as $t\to \infty$, where $\gamma:=\frac{n}{2}(q-1)-\frac{1}{2}$ and $\hat{u}(x, t)$ is defined by 
 \begin{align}
  &\begin{aligned}\label{IK-asymp}
    \hat{u}(x, t):=\ &MG(x, 1+t)-c(t)\cdot \nabla G(x, 1+t)\\
    &+|M|^{q-1}M\int_{0}^{t}\left( d\cdot \nabla G(t-\tau)*G(1+\tau)^{q} \right)(x)d\tau,
  \end{aligned}\\
  &c(t):=(c_{1}(t), c_{2}(t), \cdots, c_{n}(t)), \nonumber  \\
  &c_{i}(t):=\int_{\mathbb{R}^{n}}x_{i}u(x, t)dx-|M|^{q-1}M\int_{0}^{t}\int_{\mathbb{R}^{n}}x_{i}\left\{d\cdot \nabla\left(G(x, 1+\tau)^{q}\right)\right\}dxd\tau, \nonumber 
 \end{align}
 while $G(x, t)$ and $M$ are defined by \eqref{heat}. 
\end{prop}

Here, we remark that actually in \cite{I-K}, not only the convection-diffusion equation \eqref{CD} with $a(x)\equiv1$ but also the asymptotic behavior of solutions to the more general parabolic equations: 
\begin{equation}\label{general}
  \begin{aligned}
    &\partial_t u -  \Delta u = F(x, t, u, \nabla u),\ \  x\in\mathbb{R}^n, \ t>0,\\
    &u(x,0)=u_0(x),\ \  x\in\mathbb{R}^n
\end{aligned}
\end{equation}
with some nonlinear function $F\in C(\R^{n}\times (0, \infty)\times \R \times \R^{n})$ are studied. The above Proposition~\ref{IK} is a consequence of that general theory. For details, see the original paper \cite{I-K}.

By virtue of Proposition~\ref{IK}, when $2\gamma>1/2$, an asymptotic expansion of the solution has been accomplished up to the order $o( t^{-(n/2)(1-1/p)-1/2})$. We note that the critical case $q=1+\frac{2}{n}$ is included in this case $2\gamma>1/2$. However, we note that these asymptotic profiles obtained in Propositions~\ref{prop.DC-sub}, \ref{prop.DC-critical} and \ref{prop.DC-super}, and Proposition~\ref{IK} are different each other.
The above function $\hat{u}(x, t)$ in \eqref{IK-asymp} has the advantage of being able to express the asymptotic profile of the solution to \eqref{CD} in a unified manner for any $q>1+\frac{1}{n}$. However, comparing the results obtained in \cite{D-C, Z2} and $\hat{u}(x, t)$, the effect of the nonlinearity on the asymptotic profiles is a little bit difficult to see, especially in the critical case $q=1+\frac{2}{n}$. Indeed, the logarithmic term does not appear explicitly in $\hat{u}(x, t)$. On the other hand, in \cite{D-C, Z2}, the second asymptotic profiles are decomposed and expressed in detail depending on the nonlinear exponent $q$, but the given asymptotic rate is slower than $o( t^{-(n/2)(1-1/p)-1/2})$ by $o(\log t)$ when $q=1+\frac{2}{n}$. Taking into account this situation, we believe that it might be effective to derive a new asymptotic formula which is better in terms of both the structure of the asymptotic profile and the asymptotic rate.

In this study, we focus on the critical case $q=1+\frac{2}{n}$ and study the asymptotic behavior of the solution to \eqref{CD}, based on the above perspectives. Especially, we investigate whether the asymptotic rate $o( t^{-(n/2)(1-1/p)-1/2})$ can be realized when the second asymptotic profiles are of the form in \eqref{second}. 
Namely, we consider that whether the asymptotic rate $o( t^{-(n/2)(1-1/p)-1/2}\log t )$ given in \eqref{Z-2nd} and \eqref{Z-2nd-2} is optimal or not. 
As a result, by analyzing the corresponding integral equation \eqref{IE} below in details, we have succeeded to give the more higher-order asymptotic expansion of the solution, which generalizes Proposition~\ref{prop.DC-critical} given by Duro--Carpio \cite{D-C}. 

\medskip
The rest of this paper is organized as follows. In Section~2, we shall state our main results Theorems~\ref{mainresult-1} and \ref{mainresult-2} below. Next, we would like to give the proofs of these results in Section~3. This section is divided into two parts. Subsection~3.1 is devoted to prove Theorem~\ref{mainresult-1}. In this part, we mainly study the effect of the nonlinearity. On the other hand, Theorem~\ref{mainresult-2} will be proven in Subsection~3.2. The effect of the variable diffusion is investigated in this part. 
The main novelty of this paper is to derive the higher-order asymptotic profiles which take into account the effects of both the nonlinearity and the variable diffusion. 
The key estimates are in Propositions~\ref{prop1} and \ref{prop-new-h} below, and the main techniques of the proofs are based on the idea used in \cite{F-I} for the asymptotic analysis of a dispersive-dissipative equation which has a similar structure of \eqref{CD}.

\section{Main Results}

In this section, we would like to state our main results. 
First, let us introduce the result in the case of $n\ge2$. 
In order to do that, we shall define the following new function $\Psi(x, t)$: 
\begin{align}
  \Psi (x,t) \coloneqq t^{-\frac{n}{2} -\frac{1}{2}}\Psi _\ast \left( \frac{x}{\sqrt{t}}  \right),\ \ 
  \Psi _\ast (x) \coloneqq d \cdot \nabla \left( \int_{0}^{1} {\left( G(1-s) \ast F(s) \right) \left( x \right)}ds \right), \label{DEF-Psi}
  \end{align}
  where $F(y, s)$ is defined by 
  \begin{align}
  F(y,s) \coloneqq s^{-\frac{n}{2} -1}F_\ast \left( \frac{y}{\sqrt{s} }  \right),\ \ F_\ast (y) \coloneqq \frac{1}{2^{n+2}\pi^{\frac{n}{2} +1}} \left\{ e^{-\frac{|y|^{2}}{4}\left( 1+\frac{2}{n}  \right)} - \left( 1+\frac{2}{n}  \right)^{-\frac{n}{2}} e^{-\frac{|y|^{2}}{4} }\right\}. \label{DEF-F} 
  \end{align}
  In addition, we shall set the constants $\mathcal{M}$ and $\mathcal{N}$ by
  \begin{align}
  &\mathcal{M} \coloneqq d\int_{0}^{1} {\int_{\mathbb{R}^n} {\left( |u|^{\frac{2}{n}} u \right)}(y,\tau)dy }d\tau + d\int_{1}^{\infty} {\int_{\mathbb{R}^n} {\left( |u|^{\frac{2}{n}} u -\left\lvert MG \right\rvert ^{\frac{2}{n}} MG\right)}(y,\tau)dy }d\tau, \label{DEF-mathM} \\
   &\mathcal{N} \coloneqq \int_{0}^{\infty}\int_{\R^{n}}b(y)\nabla u(y, \tau)dyd\tau, \ n\ge2.  \ \  
    \label{DEF-mathN}
\end{align}
Then, we are able to obtain the following higher-order asymptotic expansion of the solution: 
\begin{thm}\label{mainresult-1}
Let $n\ge2$ and $ q=1+\frac{2}{n} $. Suppose $ u_0 \in L^{1}( \mathbb{R}^n; 1+|x| )\cap L^{\infty }( \mathbb{R}^n )$. 
Then, for any $ p \in [1, \infty ] $, the solution $ u(x,t) $ to \eqref{CD} satisfies the following asymptotic formula:
\begin{equation}
  \begin{aligned}
    \lim_{t \to \infty} t^{\frac{n}{2} \left( 1-\frac{1}{p} \right) +\frac{1}{2} }\biggl\lVert u(t) &- MG(t)-\alpha_n|M|^{\frac{2}{n} }M(\log t) d \cdot \nabla G(t) \\
    &- (\mathcal{M}+\mathcal{N}-m)\cdot \nabla G(t)-|M|^{\frac{2}{n} }M \Psi(t) \biggl\rVert _{L^p} = 0,  \label{main-1}
  \end{aligned}
\end{equation}
where $G(x, t)$ and $M$, $m$ and $\alpha_{n}$ are defined by \eqref{heat}, \eqref{m} and \eqref{alpha_n}, respectively. In addition, $\Psi(x, t)$, $\mathcal{M}$ and $\mathcal{N}$ are defined by \eqref{DEF-Psi}, \eqref{DEF-mathM} and \eqref{DEF-mathN}, respectively. 
\end{thm}

\begin{rem}\label{rem.self}
Compared with the general result given by Ishige--Kawakami {\rm \cite{I-K}} (i.e., the above Proposition~\ref{IK}), our formula has some advantages with respect to the structures of the asymptotic profiles. More precisely, we emphasize that the logarithmic term explicitly appears in \eqref{main-1} and the newly obtained asymptotic profile $\Psi(x, t)$ has the following parabolic self-similarity:   
\[
\Psi(x, t)=\lambda^{n+1}\Psi(\lambda x, \lambda^{2}t), \ \ \text{for} \ \ \lambda>0.
\]
\end{rem}

\begin{rem}\label{cor-optimal}
In view of the higher-order asymptotic expansion, we can obtain the optimal asymptotic rate to the second asymptotic profile, by virtue of the parabolic self-similarity of the asymptotic profiles. 
More precisely, under the same assumptions in Theorem~\ref{mainresult-1}, the solution $u(x, t)$ to \eqref{CD} satisfies the following estimate: 
  \begin{equation}\label{main2}
    \left\lVert u(t)-MG(t)-\alpha_n\left\lvert M\right\rvert ^{\frac{2}{n}}M(\log t)d\cdot\nabla G(t)\right\rVert _{L^p}=\left(C_*+o(1)\right)t^{-\frac{n}{2}\left(1-\frac{1}{p}\right)-\frac{1}{2}}
  \end{equation}
  as $ t \to \infty  $, for any $ p \in [1,\infty ] $, where the constant $C_{*}$ is defined by 
  \begin{equation}\label{C*}
  C_*\coloneqq\left\lVert \left(\mathcal{M}+\mathcal{N}-m\right)\cdot\nabla G(1) +\left\lvert M\right\rvert^{\frac{2}{n}} M\Psi_* \right\rVert _{L^p}.
  \end{equation}
  From this result, for the asymptotic rate to the second asymptotic profile, we can conclude that the optimal asymptotic rate is $O( t^{-(n/2)(1-1/p)-1/2})$ if and only if $C_{*}\neq0$. This fact implies that the asymptotic rate given in \eqref{Z-2nd} can be improved, i.e., our formula \eqref{main2} generalizes the result \eqref{Z-2nd} given in {\rm \cite{D-C}}. 
\end{rem}

\begin{rem}\label{rem-const-diffu-2}
Similarly as Remark~\ref{rem-const-diffu}, if $b(x)$ is an identically constant, then the constant $\mathcal{N}$ defined by \eqref{DEF-mathN} satisfies $\mathcal{N}=0$. Therefore, in this case, the above results \eqref{main-1} and \eqref{main2} with $\mathcal{N}=0$ are also valid for the case of $n=1$ (see, also Remark~\ref{rem-last} below). 
\end{rem}

Next, we would like to introduce the result in the case of $n=1$ and $q=3$. 
As we have already seen in \eqref{Z-2nd-2}, for deriving the asymptotic profile in this case, we need to introduce the function $K(t)$ defined by \eqref{alpha_n}. 
Moreover, in order to lead the higher-order asymptotic expansion, let us define the following new function $\Phi(x, t)$: 
\begin{align}
\Phi(x, t) \coloneqq t^{-1}\Phi _\ast \left( \frac{x}{\sqrt{t}}, t \right), \ \ \Phi_{*}(x, t)\coloneqq t^{\frac{3}{2}}\p_{x}\left( \int_{0}^{1} {\left( G(1-s) \ast \phi\left(\sqrt{t}\,\cdot, ts\right)\right) \left( x \right)} \,ds \right), \label{DEF-Phi}
\end{align}
where $\phi(y, \tau)$ is defined by 
\begin{align}
\phi(y, \tau)\coloneqq b(y)\p_{x}u(y, \tau)-\left(\int_{\R}b(z)\p_{x}u(z, \tau)dz\right)G(y, \tau). \label{DEF-phi}
\end{align}
Then, we can state the following higher-order asymptotic expansion of the solution for $n=1$:
\begin{thm}\label{mainresult-2}
Let $n=1$ and $ q=3$. Suppose $ u_0 \in L^{1}( \mathbb{R}; 1+|x| )\cap L^{\infty }( \mathbb{R} )$. 
Then, for any $ p \in [1, \infty ] $, the solution $ u(x,t) $ to \eqref{CD} satisfies the following asymptotic formula:
\begin{equation}
  \begin{aligned}
    \lim_{t \to \infty} t^{\frac{1}{2} \left( 1-\frac{1}{p} \right) +\frac{1}{2} }\biggl\lVert u(t) &- MG(t)-\left(K(t)+\alpha_{1}dM^{3}\right)(\log t)\p_{x}G(t) \\
    &- (\mathcal{M}-m)\p_{x}G(t)-M^{3}\Psi(t) -\Phi(t) \biggl\rVert _{L^p} = 0,  \label{main-2}
  \end{aligned}
\end{equation}
where $G(x, t)$ and $M$ are defined by \eqref{heat}, while $m$ is defined by \eqref{m}. In addition, $\alpha_{1}$ and $K(t)$ are defined by \eqref{alpha_n}. Moreover, $\Psi(x, t)$ and $\mathcal{M}$ are defined by \eqref{DEF-Psi} and \eqref{DEF-mathM}, respectively. Furthermore, $\Phi(x, t)$ is defined by \eqref{DEF-Phi}. 
\end{thm}

\begin{rem}
Since the functions $K(t)$ and $\Phi_{*}(x, t)$ are depending on the original solution $u(x, t)$, 
 one may consider that $K(t)(\log t)\p_{x}G(x, t)$ and $\Phi(x, t)$ in \eqref{main-2} should be called the ``correction terms'' rather than the ``asymptotic profiles''. 
 However, we are able to find their good properties as $t\to \infty$. Actually, for any $p\in [1, \infty]$, the following uniform boundedness hold: 
 \[
 \left|K(t)\right|\le C, \ \ \left\|\Phi_{*}(t)\right\|_{L^{p}}\le C, \ \ t>1. 
 \]
 For the proofs, see \eqref{N-finite} and \eqref{final-est} through \eqref{I1-I3-comb} below. 
 Therefore, taking into account these facts, 
 we can conclude that $K(t)(\log t)\p_{x}G(x, t)$ and $\Phi(x, t)$ decay at least on the following orders:
 \[
 \left\|K(t)(\log t)\p_{x}G(t)\right\|_{L^{p}}=O\left(t^{-\frac{1}{2}\left(1-\frac{1}{p}\right)-\frac{1}{2}}\log t\right), \ \ 
  \left\|\Phi(t)\right\|_{L^{p}}=O\left(t^{-\frac{1}{2}\left(1-\frac{1}{p}\right)-\frac{1}{2}}\right) \ \ \text{as} \ t\to \infty. 
 \]
These orders are the same as those for the second asymptotic profile (i.e., $\alpha_{1}dM^{3}(\log t)\p_{x}G(x, t)$) and the third asymptotic profile (i.e., $(\mathcal{M}-m)\p_{x}G(x, t)+M^{3}\Psi(x, t)$), respectively. 
\end{rem}

\begin{rem}\label{rem-last}
Similarly as Remarks~\ref{rem-const-diffu} and \ref{rem-const-diffu-2}, if $b(x)$ is an identically constant, then $K(t)$ and $\Phi(x, t)$ always vanish. Moreover, as we already mentioned in Remark~\ref{rem-const-diffu-2}, $\mathcal{N}$ also vanishes in this case. 
Therefore, under this situation, Theorems~\ref{mainresult-1} and \ref{mainresult-2} can be unified. 
More precisely, the above results \eqref{main-1} and \eqref{main2} with $\mathcal{N}=0$ are valid for all dimensions $n\ge1$. 
\end{rem}

\section{Proof of the Main Results}

In this section, we would like to prove our main results Theorems~\ref{mainresult-1} and \ref{mainresult-2}, i.e., we shall derive the asymptotic formulas \eqref{main-1} and \eqref{main-2}. In what follows, let us consider the case of $ q = 1 + \frac{2}{n} $. In order to show these results, we need to analyze the corresponding integral equation for \eqref{CD}: 
\begin{align}
  u(x,t) &= \left(G(t) \ast u_0\right) \left( x \right) + \int_{0}^{t} {\left( d \cdot \nabla G(t-\tau) \ast \left( \left\lvert u \right\rvert ^{\frac{2}{n} } u \right)(\tau) \right) \left( x \right)}d\tau \nonumber \\
   &\ \ \ \ +\sum_{i=1}^{n}\int_{0}^{t}\left(\frac{\p G}{\p x_{i}}(t-\tau)\ast b\frac{\p u}{\p x_{i}}(\tau) \right)(x)d\tau \nonumber \\
   &\eqqcolon \left(G(t) \ast u_0\right) \left( x \right) + D_{1}(x, t) + D_{2}(x, t). \label{IE}
  \end{align}
Then, in order to complete the proof of Theorems~\ref{mainresult-1} and \ref{mainresult-2}, we shall derive the leading term for the all terms in the right hand side of \eqref{IE}. 
We start with introducing the basic properties for the heat kernel $G(x, t)$. First, let us recall the $L^{p}$-decay estimate of it. Let $\alpha \in \mathbb{Z}_{+}^{n}$ be a multi-index and $ p \in [1, \infty ] $. Then, the following estimate holds (for the proof, see e.g.~\cite{G-S}): 
\begin{equation}\label{Gdecay}
  \left\lVert \partial^\alpha_x G(t) \right\rVert  _{L^p} \leq Ct^{-\frac{n}{2}\left(1-\frac{1}{p}\right)-\frac{|\alpha|}{2}},\ \  t>0. 
\end{equation}
Moreover, we would like to introduce the following asymptotic formula for $\left(G(t) \ast u_0\right)(x)$. For the proof, let us refer to \cite{DZ92}. This formula plays an important role in the proof of main results. 
\begin{prop}[\cite{DZ92}]\label{linereval}
  Let $u_0 \in L^1 (\mathbb{R}^n;1+\left\lvert x\right\rvert)$ and $ p \in [1, \infty ] $. Then, we have 
  \begin{equation}\label{HEasymp2}
    \lim_{t\to \infty}t^{\frac{n}{2}\left(1-\frac{1}{p}\right)+\frac{1}{2}}\left\| G(t)*u_0-MG(t)+m\cdot\nabla G(t)\right\|_{L^p}=0,
  \end{equation}
  where $ G(x, t) $ and $M$ are defined by \eqref{heat}, while $ m $ is defined by \eqref{m}.
\end{prop}
By virtue of this proposition, to complete the proofs of the main results, we only need to consider the asymptotic behavior of the Duhamel term in \eqref{IE}. 
In what follows, we shall derive the asymptotic profiles for $D_{1}(x, t)$ and $D_{2}(x, t)$ in Subsections 3.1 and 3.2, respectively. 

\subsection{Proof of Theorem~\ref{mainresult-1}}

In this subsection, we would like to prove Theorem~\ref{mainresult-1}. 
First, let us consider the asymptotic behavior of $D_{1}(x, t)$ in \eqref{IE}. 
From the results obtained in Zuazua~\cite{Z2} and Duro--Carpio~\cite{D-C}, we can see that the asymptotic profile of $ D_{1}(x, t) $ is given by $ \alpha_n|M|^{\frac{2}{n} }M(\log t)d \cdot \nabla G(x, t)$. 
In the present paper, we construct the second asymptotic profile of $ D_{1}(x,t) $ and generalize the results given in \cite{D-C, Z2}. 
More precisely, the following asymptotic formula can be established: 
\begin{thm}\label{duhameleval}
Let $n\ge1$ and $ q=1+\frac{2}{n} $. Suppose $ u_0 \in L^{1}( \mathbb{R}^n; 1+|x| )\cap L^{\infty }( \mathbb{R}^n )$. 
Then, for any $ p \in [1, \infty ] $, the following asymptotic formula holds: 
\begin{equation}\label{D1-asymp}
  \begin{aligned}
    \lim_{t \to \infty} t^{\frac{n}{2} \left( 1-\frac{1}{p} \right)+\frac{1}{2}}\biggl\|D_{1}(t)-\alpha_n|M|^{\frac{2}{n} }M\left(\log t\right) d \cdot \nabla G(t)-|M|^{\frac{2}{n} }M \Psi(t) - \mathcal{M}\cdot \nabla G(t)\biggl\|_{L^p} = 0,  
  \end{aligned}
  \end{equation}
where $D_{1}(x, t)$ is defined by \eqref{IE}, while $G(x, t)$ and $M$, and $\alpha_{n}$ are defined by \eqref{heat} and \eqref{alpha_n}, respectively. In addition, $\Psi(x, t)$, $\mathcal{M}$ and $\mathcal{N}$ are defined by \eqref{DEF-Psi}, \eqref{DEF-mathM} and \eqref{DEF-mathN}, respectively. 
\end{thm}
\begin{rem}
For the proof of this theorem, we note that the assumption $n\ge2$ is unnecessary. 
This restriction of the dimension $n\ge2$ will be used in Theorem~\ref{thm-new1} below. 
\end{rem}
To prove this theorem, let us introduce the following new functions $ v(x, t) $ and $ w(x, t) $:
\begin{align}
  &v(x,t) := \int_{1}^{t} {\left( d \cdot \nabla G(t-\tau) \ast \left( |G|^{\frac{2}{n}} G\right)(\tau) \right) (x)}d\tau,\label{v}\\ 
  \begin{split}\label{w}
    &w(x, t) := \int_{0}^{1} {\left( d \cdot \nabla G(t-\tau) \ast \left( |u|^{\frac{2}{n}} u\right)(\tau) \right) ( x )}d\tau \\
      &\ \ \ \ \ \ \ \ \ \ \ \ \ + \int_{1}^{t} {\left( d \cdot \nabla G(t-\tau) \ast \left( |u|^{\frac{2}{n}} u - |MG|^{\frac{2}{n}}MG\right)(\tau) \right) ( x )}d\tau. 
      \end{split}
\end{align}
By using the functions $v(x,t)$ and $w(x,t)$ above, $D_{1}(x,t)$ in \eqref{IE} can be split as follows: 
\begin{align*}
  D_{1}(x, t) &= \int_{0}^{t} {\left( d \cdot \nabla G(t-\tau) \ast \left( \left\lvert u \right\rvert ^{\frac{2}{n} } u \right)(\tau) \right) \left( x \right)}d\tau\\
  &= |M|^{\frac{2}{n} }M\int_{1}^{t} {\left( d \cdot \nabla G(t-\tau) \ast \left( |G|^{\frac{2}{n}} G\right)(\tau) \right) \left( x \right)}d\tau\\
  &\ \ \ \ +\int_{0}^{1} {\left( d \cdot \nabla G(t-\tau) \ast \left( |u|^{\frac{2}{n}} u\right)(\tau) \right) \left( x \right)}d\tau \\
  &\ \ \ \ + \int_{1}^{t} {\left( d \cdot \nabla G(t-\tau) \ast \left( |u|^{\frac{2}{n}} u - |MG|^{\frac{2}{n}}MG\right)(\tau) \right) \left( x \right)}d\tau = |M|^{\frac{2}{n} }Mv(x, t) + w(x, t).
 \end{align*}
 Then, from the above expression, we can see that the following relation holds:
  \begin{align}
    &D_{1}(x ,t)-\alpha_n|M|^{\frac{2}{n} }M(\log t)d \cdot \nabla G(x ,t) -|M|^{\frac{2}{n} }M\Psi(x ,t) - \mathcal{M}\cdot \nabla G(x ,t)\nonumber \\
    &=|M|^{\frac{2}{n} }M\left\{ v(x, t) - \alpha_n(\log t) d \cdot \nabla G(x, t) - \Psi (x, t) \right\} + \left\{ w(x, t) - \mathcal{M} \cdot \nabla G(x, t) \right\}. \label{D-split}
  \end{align}
Therefore, in order to complete the proof of Theorem~\ref{duhameleval}, we need to evaluate the first and the second terms of the above equation. 
First, let us derive the following estimate for the first term: 
\begin{prop}\label{prop1}
  Let $ n\ge1$. Then, for any $ p \in [1,\infty ]$, the following asymptotic formula holds: 
  \begin{equation}\label{v-asymp}
    \left\lVert v(t) - \alpha_n(\log t) d \cdot \nabla G(t) - \Psi (t) \right\rVert _{L^p} \le Ct^{-\frac{n}{2} \left( 1-\frac{1}{p} \right)-1}, \ \ t\ge2,
    \end{equation}
    where $ v(x,t) $, $ G(x,t) $, $ \Psi (x,t) $ and $\alpha_{n}$ are defined by \eqref{v}, \eqref{heat}, \eqref{DEF-Psi} and \eqref{alpha_n}, respectively.
\end{prop}
\begin{proof}
  First, we set $ V(x, t) \coloneqq \alpha _n (\log t) d \cdot \nabla G(x, t)$. Then, it follows that
  \begin{equation*}
    \begin{aligned}
      \partial _t V(x, t) &= \frac{ \alpha _n }{t}\left( d \cdot \nabla G(x,t) \right) + \alpha _n (\log t) \partial _t \left( d \cdot \nabla G(x,t) \right) \\
      &= \frac{ \alpha _n }{t} \left( d \cdot \nabla G(x,t) \right) + \alpha _n (\log t)\Delta \left( d \cdot \nabla G(x,t) \right). 
    \end{aligned}
  \end{equation*}
  Thus, we can see that $V(x, t)$ is the solution to the following Cauchy problem:
  \begin{equation}\label{CP-V}
    \begin{aligned}
      &\partial _t V - \Delta V = \frac{\alpha _n}{t} \left( d \cdot \nabla G(x, t) \right), \ \ x\in \R^{n}, \ t>1, \\
      &V(x, 1) = 0, \ \ x\in \R^{n}.
    \end{aligned}
  \end{equation}
  Applying the Duhamel principle, $ V(x,t) $ can be rewritten by
    \begin{align*}
      V(x, t) 
      &= \int_{1}^{t} {\left( d \cdot \nabla G(t-\tau) \ast \alpha _n \tau^{-1}G(\tau) \right)(x)}d\tau.
    \end{align*}
  Now, it follows from \eqref{heat} and \eqref{DEF-F} that $( |G|^{\frac{2}{n}} G - \alpha _n \tau^{-1}G )(y, \tau)=F(y, \tau)$. 
  Therefore, using the above expression and \eqref{v}, we are able to obtain 
    \begin{align}
      v(x, t) - V(x, t)
            &= \int_{1}^{t} {\left( d \cdot \nabla G(t-\tau) \ast \left( |G|^{\frac{2}{n}} G - \alpha _n \tau^{-1}G\right)(\tau) \right) (x)}d\tau \nonumber \\
      &= d \cdot \nabla \int_{1}^{t} \int_{\mathbb{R}^n} G(x-y, t-\tau) F(y, \tau)dyd\tau =: d \cdot \nabla I(x, t). \label{v-V}
    \end{align}
     
  Next, let us transform $ I(x, t) $ in the above equation \eqref{v-V}, by using the scaling argument. 
  Actually, from \eqref{heat} and \eqref{DEF-F}, we note that the following self-similar structures of $G(x, t)$ and $F(y, \tau)$ hold true:  
 \begin{equation*}\label{scal-1}
 G(x, t)=\lambda^{n}G(\lambda x, \lambda^{2}t), \ \ F(y, \tau)=\lambda^{n+2}F(\lambda y, \lambda^{2}\tau), \ \ \text{for} \ \ \lambda>0. 
 \end{equation*} 
  By virtue of the above properties, we can easily obtain 
    \begin{equation}\label{scal-2}
    G(x-y, t-\tau)=t^{-\frac{n}{2}}G\left(\frac{x-y}{\sqrt{t}}, 1-\frac{\tau}{t}\right), \ \ F(y, \tau)=t^{-\frac{n}{2}-1}F\left(\frac{y}{\sqrt{t}}, \frac{\tau}{t}\right). 
    \end{equation}
  Therefore, using the above equations and the change of variables, we obtain 
  \begin{align}
    I(x, t) &=t^{-n-1}\int_{1}^{t}\int_{\R^{n}}G\left(\frac{x-y}{\sqrt{t}}, 1-\frac{\tau}{t}\right)F\left(\frac{y}{\sqrt{t}}, \frac{\tau}{t}\right)dyd\tau \nonumber \\
        &= t^{-\frac{n}{2} } \int_{\frac{1}{t}}^{1} \int_{\mathbb{R}^n}G\left(\frac{x}{\sqrt{t}} - z, 1 - s \right) F(z, s)dzds\nonumber \\
    &= t^{-\frac{n}{2} } \int_{\frac{1}{t}}^{1} {\left( G(1-s) \ast F(s) \right) \left( \frac{x}{\sqrt{t}} \right)}ds. \label{K-henkei}
  \end{align}
  Now, recalling the definition of $\Psi(x, t)$ in \eqref{DEF-Psi}, it can be rewritten as follows:
  \begin{align}
    \Psi(x, t) &= t^{-\frac{n}{2}-\frac{1}{2}}\Psi_\ast\left( \frac{x}{\sqrt{t}} \right) = t^{-\frac{n}{2}-\frac{1}{2}} d \cdot \nabla \left( \int_{0}^{1} {\left( G(1-s) \ast F(s) \right) \left( x \right)}ds \right)\biggl|_{x=\frac{x}{\sqrt{t}}} \nonumber \\
    &=t^{-\frac{n}{2}} d \cdot \nabla \left(\int_{0}^{1} {\left( G(1-s) \ast F(s) \right) \left( \frac{x}{\sqrt{t}} \right)}ds\right). \label{Psi-t}
  \end{align}
  Thus, for any $p\in [1, \infty]$, it follows from \eqref{v-V} and \eqref{K-henkei} that 
  \begin{align}
    \left\lVert v(t) - V(t) - \Psi (t) \right\rVert _{L^p} 
    &= t^{-\frac{n}{2}} \left\lVert d \cdot \nabla \left(\int_{0}^{\frac{1}{t}} {\left( G(1-s) \ast F(s) \right) \left( \frac{\cdot }{\sqrt{t}} \right)}ds\right)\right\rVert _{L^p} \nonumber \\
    &= t^{-\frac{n}{2}\left(1-\frac{1}{p}\right)-\frac{1}{2}} \left\lVert d \cdot \nabla \int_{0}^{\frac{1}{t}} {G(1-s) \ast F(s)}ds\right\rVert _{L^p}. \label{mainpart}
  \end{align}
  
  In what follows, we shall evaluate the right hand side of the above \eqref{mainpart}. 
  Now, noticing $\int_{\mathbb{R}^n}{F_\ast \left( z \right)}dz = 0 $. Therefore, using the mean value theorem, \eqref{scal-2}, the Schwarz inequality, the Young inequality and \eqref{Gdecay}, we obtain 
  \begin{align}
    &\left\lVert d \cdot \nabla \int_{0}^{\frac{1}{t}} {G(1-s) \ast F(s)}ds\right\rVert _{L^p} \nonumber \\
    &\le \int_{0}^{\frac{1}{t}} {\left\lVert d \cdot \nabla \int_{\mathbb{R}^n}{G(\cdot - y, 1-s)s^{-\frac{n}{2}-1} F_\ast \left( \frac{y}{\sqrt{s}} \right)}dy  \right\rVert _{L^p}}ds \nonumber\\
    &= \int_{0}^{\frac{1}{t}} {s^{-\frac{n}{2}-1}\left\lVert d \cdot \nabla \int_{\mathbb{R}^n}^{} {\left(\int_{0}^{1} {y \cdot \nabla G(\cdot-\theta y, 1-s)}d\theta\right)F_\ast \left( \frac{y}{\sqrt{s}} \right)}dy \right\rVert _{L^p}}ds \nonumber\\
    &= \int_{0}^{\frac{1}{t}} {s^{-\frac{n}{2}-1}\left\lVert\int_{\mathbb{R}^n}^{} {\left(\int_{0}^{1} {d \cdot \left(D^2 G(\cdot-\theta y, 1-s)y\right)}d\theta\right)} F_\ast \left( \frac{y}{\sqrt{s}} \right)dy \right\rVert _{L^p}} ds \nonumber\\
    &\le \left\lvert d \right\rvert\int_{0}^{\frac{1}{t}} {s^{-\frac{n}{2}-\frac{1}{2} }\int_{0}^{1}\left\lVert\ {\frac{1}{\theta ^n} \int_{\mathbb{R}^n}^{} {\left\lvert D^2 G\left( \frac{\cdot}{\theta} - y, \frac{1-s}{\theta ^2} \right) \right\rvert \frac{\left\lvert y \right\rvert}{\sqrt{s} }  \left\lvert F_\ast \left( \frac{y}{\sqrt{s}} \right) \right\rvert} \,dy}\right\rVert _{L^p}}d\theta ds \nonumber \\
    &\leq C\int_{0}^{\frac{1}{t}} {s^{-\frac{n}{2}-\frac{1}{2} }\int_{0}^{1} {\theta^{-n\left( 1-\frac{1}{p}  \right)-2} (1-s)^{-\frac{n}{2} \left( 1-\frac{1}{p} \right) -1}\theta^{n\left( 1-\frac{1}{p} \right)+2}s^{\frac{n}{2}} \left\lVert |\cdot|F _\ast \right\rVert _{L^1}}d\theta}ds \nonumber\\
    &\leq C\left\lVert |\cdot|F _\ast \right\rVert _{L^1}\int_{0}^{\frac{1}{t}} {s^{-\frac{1}{2} }(1-s)^{-\frac{n}{2} \left( 1-\frac{1}{p} \right) -1}}ds 
    \leq C\left\lVert |\cdot|F _\ast \right\rVert _{L^1}t^{-\frac{1}{2}}, \ \ t\ge2, \label{dag-est}
  \end{align}
    for any $p\in [1, \infty]$. Finally, combining \eqref{mainpart} and \eqref{dag-est}, we can conclude that the desired estimate \eqref{v-asymp} is true. This completes the proof. 
\end{proof}

Next, we would like to treat the second term in the right hand side of \eqref{D-split}. In order to do that, we prepare the following lemma: 
\begin{lemma}\label{lemma}
Let $n\ge1$ and $ q=1+\frac{2}{n} $. Suppose $ u_0 \in L^{1}( \mathbb{R}^n; 1+|x| )\cap L^{\infty }( \mathbb{R}^n )$. 
Then, for any $ p \in [1, \infty ] $, the solution $u(x, t)$ to \eqref{CD} satisfies the following estimate: 
  \begin{equation}\label{uginequality}
    \left\lVert \left( |u|^{\frac{2}{n}}u \right)(t) - \left( |MG|^{\frac{2}{n}}MG \right)(t) \right\rVert _{L^p} \leq Ct^{-\frac{n}{2}\left( 1-\frac{1}{p} \right)-\frac{3}{2}}\log(2+t), \ \ t\ge1, 
  \end{equation}
  where $G(x, t)$ and $M$ are defined by \eqref{heat}. 
\end{lemma}
\begin{proof}
  First, we recall that the following inequality holds:
  \begin{equation*}
    \left\lvert |\alpha|^{q-1}\alpha - |\beta|^{q-1}\beta \right\rvert \leq C\left( |\alpha|^{q-1} + |\beta|^{q-1} \right) \left\lvert \alpha - \beta \right\rvert, \ \  \alpha,\beta \in \mathbb{R},\  q>1. 
  \end{equation*}
  Therefore, it follows from \eqref{CDdecay}, \eqref{CDasymp} and \eqref{Gdecay} that 
  \begin{align*}
    &\left\lVert \left( |u|^{\frac{2}{n}}u \right)(t) - \left( |MG|^{\frac{2}{n}}MG \right)(t) \right\rVert _{L^p} \\
    &\leq C \left( \left\lVert u(t) \right\rVert _{L^\infty}^{\frac{2}{n}} + \left\lVert MG(t) \right\rVert _{L^\infty}^{\frac{2}{n}} \right)\left\lVert u(t) - MG(t) \right\rVert _{L^p}\\
    &\leq C \left( \left( t^{-\frac{n}{2}} \right)^{\frac{2}{n}} + \left( |M|t^{-\frac{n}{2}} \right)^{\frac{2}{n}} \right) t^{-\frac{n}{2}\left( 1-\frac{1}{p} \right)-\frac{1}{2}}\log (2+t) \\
    &\leq Ct^{-\frac{n}{2}\left( 1-\frac{1}{p} \right)-\frac{3}{2}}\log (2+t), \ \ t\ge1.
  \end{align*}
  This completes the proof.
\end{proof}

In what follows, let us deal with the second term in the right hand side of \eqref{D-split}, i.e., we would like to introduce the asymptotic profile for $ w(x, t) $. 
Actually, the following formula can be established: 
\begin{prop}\label{prop2}
Let $n\ge1$ and $ q=1+\frac{2}{n} $. Suppose $ u_0 \in L^{1}( \mathbb{R}^n; 1+|x| )\cap L^{\infty }( \mathbb{R}^n )$. 
Then, for any $ p \in [1, \infty ] $, the following asymptotic formula holds: 
   \begin{equation}\label{w-asymp}
    \lim_{t \to \infty} t^{\frac{n}{2} \left( 1-\frac{1}{p} \right)+\frac{1}{2}} \left\lVert w(t) - \mathcal{M}\cdot \nabla G(t) \right\rVert _{L^p} = 0,
   \end{equation}
   where $ w(x,t) $, $ G(x,t) $ and $ \mathcal{M} $ are defined by \eqref{w}, \eqref{heat} and \eqref{DEF-mathM}, respectively.
\end{prop}
\begin{proof}
  First, for simplicity, let us define $ \rho(x, t) $, $ \mathcal{M}_0 $ and $ \mathcal{M}_1 $ as follows: 
  \begin{align}
    &\rho(x, t) \coloneqq \left( |u|^{\frac{2}{n}}u - |MG|^{\frac{2}{n}}MG \right)(x, t), \label{rho}\\
     &\mathcal{M}_0 \coloneqq d\int_{0}^{1} {\int_{\mathbb{R}^n}^{} {\left( |u|^\frac{2}{n} u \right)(y, \tau)}dy}d\tau, \ \  \mathcal{M}_1 \coloneqq d\int_{1}^{\infty } {\int_{\mathbb{R}^n}^{} {\rho(y ,\tau)}dy }d\tau. \label{M0M1}
  \end{align}
  Using the definitions of $\rho(x, t)$ in \eqref{rho}, $ \mathcal{M}_0 $ and $ \mathcal{M}_1 $ in \eqref{M0M1}, and $ \mathcal{M} $ in \eqref{DEF-mathM}, we have
  \begin{equation*}
    \mathcal{M} = d\int_{0}^{1} {\int_{\mathbb{R}^n} {\left( |u|^\frac{2}{n} u \right)(y, \tau)}dy}d\tau + d\int_{1}^{\infty } {\int_{\mathbb{R}^n} {\rho(y ,\tau)}dy }d\tau = \mathcal{M}_0 + \mathcal{M}_1.
  \end{equation*}
  Then, it follows from the definition of $ w(x,t) $ in \eqref{w} that 
    \begin{align}
     &w(x, t) - \mathcal{M} \cdot \nabla G(x, t) \nonumber \\
     &= w(x, t) - \mathcal{M}_0 \cdot \nabla G(x, t) - \mathcal{M}_1 \cdot \nabla G(x, t) \nonumber \\
     &=\left\lbrace \int_{1}^{t} {\left( d \cdot \nabla G(t-\tau) \ast \rho(\tau) \right) \left( x \right)}d\tau - \mathcal{M}_1\cdot \nabla G(x, t) \right\rbrace \nonumber \\
     &\ \ \ \ +\left\lbrace \int_{0}^{1} {\left( d \cdot \nabla G(t-\tau) \ast \left( |u|^{\frac{2}{n}} u\right)(\tau) \right) \left( x \right)}d\tau - \mathcal{M}_0 \cdot \nabla G(x, t)\right\rbrace. \label{w-split}
    \end{align}
  Therefore, in order to prove \eqref{w-asymp}, it is sufficient to derive the following asymptotic formulas:
  \begin{align}
    &\lim_{t \to \infty} t^{\frac{n}{2} \left( 1-\frac{1}{p} \right)+\frac{1}{2}} \left\lVert \int_{1}^{t} { d \cdot \nabla G(t-\tau) \ast \rho(\tau) } d\tau - \mathcal{M}_1\cdot \nabla G(t)\right\rVert _{L^p} = 0,\label{prop3-1}\\
    &\lim_{t \to \infty} t^{\frac{n}{2} \left( 1-\frac{1}{p} \right)+\frac{1}{2}} \left\lVert\int_{0}^{1} {d \cdot \nabla G(t-\tau) \ast \left( |u|^{\frac{2}{n}} u\right)(\tau) } d\tau - \mathcal{M}_0\cdot \nabla G(t)\right\rVert _{L^p} = 0, \label{prop3-2}
  \end{align}
  for any $p\in [1, \infty]$. First, we shall prove \eqref{prop3-1}. From the definition of $\mathcal{M}_{1}$ in \eqref{M0M1}, we have
  \begin{align}
    &\int_{1}^{t} { \left( d \cdot \nabla G(t-\tau) \ast \rho(\tau) \right) \left( x \right) }d\tau - \mathcal{M}_1\cdot \nabla G(x, t)\nonumber\\
    &= \int_{1}^{t} {\int_{\mathbb{R}^n} {d \cdot \nabla G(x-y, t-\tau) \rho(y, \tau)}dy }d\tau - d \cdot \nabla G(x, t)\int_{1}^{\infty } {\int_{\mathbb{R}^n} {\rho(y ,\tau)} dy } d\tau\nonumber\\
    &= \int_{1}^{t} {\int_{\mathbb{R}^n}{d \cdot \nabla (G(x-y, t-\tau) - G(x, t)) \rho(y, \tau)} dy } d\tau - d \cdot \nabla G(x, t)\int_{t}^{\infty } {\int_{\mathbb{R}^n} {\rho(y ,\tau)}dy }d\tau\nonumber\\
    &\eqqcolon X(x, t) + Y(x, t).\label{xy-split}
   \end{align}
   In addition, let us take $ \varepsilon \in (0, 1)$ and split the integral of $X(x, t)$ as follows: 
   \begin{align}
    X(x, t) 
    &= \int_{\frac{\varepsilon t}{2}}^{t} {\int_{\mathbb{R}^n}{d \cdot \nabla (G(x-y, t-\tau) - G(x, t)) \rho(y, \tau)}dy }d\tau\nonumber\\ 
    &\ \ \ \ + \int_{1}^{\frac{\varepsilon t}{2}} {\left( \int_{|y| \ge \varepsilon \sqrt{t}} +\int_{|y| \le \varepsilon \sqrt{t}} \right) {d \cdot \nabla (G(x-y, t-\tau) - G(x, t)) \rho(y, \tau)}dy }d\tau\nonumber\\
    &\eqqcolon X_1(x, t) + X_2(x, t) + X_3(x, t).\label{x-split}
   \end{align}
   
   In what follows, we would like to evaluate $ X_1(x, t) $, $ X_2(x, t) $ and $ X_3(x, t) $. Here and later, let $p\in [1, \infty]$. 
   First, from the definition of $ \rho(y, \tau) $ in \eqref{rho} and Lemma~\ref{lemma}, we obtain
   \begin{equation}\label{rho-eval}
    \left\lVert \rho(\tau) \right\rVert _{L^p} = \left\lVert \left( |u|^{\frac{2}{n}}u \right)(\tau) - \left( |MG|^{\frac{2}{n}}MG \right)(\tau) \right\rVert _{L^p} \leq C\tau^{-\frac{n}{2}\left( 1-\frac{1}{p} \right)-\frac{3}{2}}\log (2+\tau), \ \tau \ge1.
   \end{equation}
   Now, we shall treat $X_{1}(x, t)$. From the Young inequality, \eqref{Gdecay} and \eqref{rho-eval}, we have
   \begin{align}
    \left\lVert X_1(t) \right\rVert _{L^p}
     &\leq |d|\int_{\frac{\varepsilon t}{2}}^{t} {\left\lVert\nabla G(t-\tau) \right\rVert _{L^1}} \left\lVert \rho(\tau) \right\rVert _{L^p}d\tau + |d|\left\lVert\nabla G(t) \right\rVert _{L^p}\int_{\frac{\varepsilon t}{2}}^{t} {\left\lVert \rho(\tau) \right\rVert _{L^1}} d\tau\nonumber\\
     &\leq C\int_{\frac{\varepsilon t}{2}}^{t} {(t-\tau)^{-\frac{1}{2}}\tau^{-\frac{n}{2}\left( 1-\frac{1}{p}\right)-\frac{3}{2}}\log (2+\tau)}d\tau + Ct^{-\frac{n}{2}\left( 1-\frac{1}{p}\right)-\frac{1}{2}}\int_{\frac{\varepsilon t}{2}}^{t} {\tau^{-\frac{3}{2}}\log (2+\tau)} d\tau\nonumber\\
     &\leq C\left( \varepsilon^{-\frac{n}{2}\left( 1-\frac{1}{p} \right)-\frac{3}{2}}+ \varepsilon^{-\frac{3}{2}} \right)t ^{-\frac{n}{2}\left( 1-\frac{1}{p} \right)-1} \log (2+t), \ \ t > \frac{2}{\varepsilon}.  \label{X1}
   \end{align}
   Next, let us deal with $ X_2(x, t) $. From \eqref{Gdecay}, we can easily obtain 
   \begin{align}
    \left\lVert X_2(t) \right\rVert _{L^p} 
    &\le |d|\int_{1}^{\frac{\varepsilon t}{2}} {(\left\lVert\nabla G(t-\tau) \right\rVert _{L^p} + \left\lVert\nabla G(t)  \right\rVert _{L^p})\int_{|y| \ge \varepsilon \sqrt{t}}^{} {\left\lvert \rho(y, \tau) \right\rvert} dy } d\tau\nonumber\\\
    &\leq C\int_{1}^{\frac{\varepsilon t}{2}}\left( (t-\tau)^{-\frac{n}{2}\left( 1-\frac{1}{p} \right)-\frac{1}{2}} + t^{-\frac{n}{2}\left( 1-\frac{1}{p} \right)-\frac{1}{2}} \right) {\int_{|y| \ge \varepsilon \sqrt{t}}^{} {\left\lvert \rho(y, \tau) \right\rvert} dy } d\tau\nonumber\\\
    &\le C t^{-\frac{n}{2}\left( 1-\frac{1}{p} \right)-\frac{1}{2}} z(t), \ \ t>\frac{2}{\varepsilon}, \label{X2}
 \end{align}
   where the function $ z(t) $ is defined by
   \begin{equation*}
    z(t) \coloneqq \int_{1}^{\frac{\varepsilon t}{2}} {\int_{|y| \ge \varepsilon \sqrt{t}}^{} {|\rho(y, \tau)|} dy } d\tau = \int_{1}^{\infty } {\chi_{\left[ 1, \frac{\varepsilon t}{2} \right]}(\tau)\int_{|y| \ge \varepsilon \sqrt{t}} {|\rho(y, \tau)|} dy } d\tau. 
  \end{equation*}
Moreover, it follows from \eqref{rho-eval} that $  \int_{1}^{\infty}\int_{\R^{n}}{|\rho(y, \tau)|}dyd\tau<\infty$. 
Therefore, from the Lebesgue dominated convergence theorem, we obtain
  \begin{align}
    \lim_{t \to \infty}z(t) &= \int_{1}^{\infty} {\lim_{t \to \infty}{\left( \chi_{\left[ 1, \frac{\varepsilon t}{2} \right]}(\tau)\int_{|y| \ge \varepsilon \sqrt{t}}^{} {|\rho(y, \tau)|}dy  \right)} d\tau} = 0. \label{zlim}
  \end{align}
  
  Finally, we would like to evaluate $ X_3(x, t) $. First, recalling $ \partial_t G = \Delta G $ and using \eqref{Gdecay}, the following estimate holds for $ 1 \le \tau \le \frac{\varepsilon t}{2} $:
    \begin{align}
     &\left\lVert d \cdot \nabla \left( G(t-\tau)-G(t) \right)\right\rVert _{L^p} \nonumber \\
     &= \tau \left\lVert d \cdot \nabla \int_{0}^{1} {\partial_t G(t-\theta\tau)} d\theta \right\rVert _{L^p} \leq \tau\int_{0}^{1} {\left\lVert d \cdot \nabla \Delta G(t-\theta\tau) \right\rVert _{L^p}} d\theta \nonumber \\
     &\leq C\tau \int_{0}^{1} {(t-\theta\tau)^{-\frac{n}{2}\left( 1-\frac{1}{p} \right)-\frac{3}{2}}} d\theta 
     \leq C\varepsilon t^{-\frac{n}{2}\left( 1-\frac{1}{p} \right)-\frac{1}{2}}. \label{X3part1}
    \end{align}
   Moreover, under this situation, if we also assume $ |y| \leq \varepsilon \sqrt{t} $, the following estimate holds too: 
   \begin{align}
    &\left\lVert d \cdot \nabla \left( G(\cdot-y, t-\tau)-G(\cdot, t-\tau) \right)\right\rVert _{L^p} 
    \leq |d||y|\int_{0}^{1} {\left\lVert D^2G(\cdot-\theta y, t-\tau) \right\rVert _{L^p}} d\theta     \nonumber \\
    &\leq C\varepsilon \sqrt{t}(t-\tau)^{-\frac{n}{2}\left( 1-\frac{1}{p} \right)-1} 
    \leq C\varepsilon t^{-\frac{n}{2}\left( 1-\frac{1}{p} \right)-\frac{1}{2}}\label{X3part2}.
   \end{align}
   Thus, from \eqref{X3part1} and \eqref{X3part2}, we obtain 
    \begin{align}
     &\left\lVert d \cdot \nabla \left( G(\cdot-y, t-\tau)-G(\cdot, t) \right)\right\rVert _{L^p} \nonumber \\
     &\leq \left\lVert d \cdot \nabla \left( G(\cdot-y, t-\tau)-G(\cdot, t-\tau) \right)\right\rVert _{L^p} + \left\lVert d \cdot \nabla \left( G(\cdot, t-\tau)-G(\cdot, t) \right)\right\rVert _{L^p} \nonumber\\
     &\le C\varepsilon t^{-\frac{n}{2}\left( 1-\frac{1}{p} \right)-\frac{1}{2}}, \ \ |y| \le \varepsilon \sqrt{t}, \ 1 \le \tau \le \frac{\varepsilon t}{2}. \label{Gdiff}
    \end{align}
    Therefore, combining the fact $  \int_{1}^{\infty}\int_{\R^{n}}{|\rho(y, \tau)|}dyd\tau<\infty$ and \eqref{Gdiff}, we are able to see that 
    \begin{align}
     \left\lVert X_3(t) \right\rVert _{L^p} &\leq \int_{1}^{\frac{\varepsilon t}{2}} {\int_{|y| \leq \varepsilon \sqrt{t}}^{} {\left\lVert d \cdot \nabla (G(\cdot-y, t-\tau) - G(\cdot, t)) \right\rVert _{L^p}}|\rho(y, \tau)| dy } d\tau \nonumber \\
     &\leq C\varepsilon t^{-\frac{n}{2}\left( 1-\frac{1}{p} \right)-\frac{1}{2}} \int_{1}^{\frac{\varepsilon t}{2}} {\int_{|y| \leq \varepsilon \sqrt{t}}^{} {|\rho(y, \tau)|} dy }d\tau \leq C\varepsilon t^{-\frac{n}{2}\left( 1-\frac{1}{p} \right)-\frac{1}{2}}, \ \ t > \frac{2}{\varepsilon}. \label{X3}
    \end{align}
   
   Next, let us evaluate $ Y(x, t) $. It follows from \eqref{Gdecay} and \eqref{rho-eval} that 
    \begin{align}
     \left\lVert Y(t) \right\rVert _{L^p} 
     &\leq |d|\left\lVert \nabla G(t) \right\rVert _{L^p}\int_{t}^{\infty } {\left\lVert \rho(\tau) \right\rVert _{L^1}} d\tau 
      \le Ct^{-\frac{n}{2}\left(1-\frac{1}{p}\right)-\frac{1}{2}}\int_{t}^{\infty}\tau^{-\frac{3}{2}}\log (2+\tau) d\tau \nonumber \\
      &=Ct^{-\frac{n}{2}\left(1-\frac{1}{p}\right)-\frac{1}{2}}\left( -2\left[ \tau^{-\frac{1}{2}}\log (2+\tau) \right]_{t}^{\infty} + 2 \int_{t}^{\infty} {\tau^{-\frac{1}{2}}(2+\tau)^{-1}} d\tau \right) \nonumber \\
     &\leq Ct ^{-\frac{n}{2}\left( 1-\frac{1}{p} \right)-\frac{1}{2}}\cdot Ct^{-\frac{1}{2}} \left\{\log (2+t) + 1\right\} 
     \leq Ct ^{-\frac{n}{2}\left( 1-\frac{1}{p} \right)-1}\log (2+t), \ \ t>1. \label{Y}
    \end{align}
   
   Summarizing up \eqref{xy-split}, \eqref{x-split}, \eqref{X1}, \eqref{X2}, \eqref{zlim}, \eqref{X3} and \eqref{Y},  
   we eventually obtain the following estimate:
   \begin{equation*}
    \limsup _{t \to \infty}{t^{\frac{n}{2}\left( 1-\frac{1}{p} \right)+\frac{1}{2}}\left\lVert \int_{1}^{t} {d \cdot \nabla G(t-\tau) \ast \rho(\tau) }d\tau - \mathcal{M}_1\cdot \nabla G(t) \right\rVert _{L^p}} \le C\varepsilon.
  \end{equation*}
  Thus, we obtain \eqref{prop3-1} because $ \varepsilon>0 $ can be chosen arbitrarily small. 
  Here, we remark that \eqref{prop3-2} can be shown by the similar way. Actually, by virtue of the regularity property of the solution $u \in C([0, \infty); L^1(\mathbb{R}^n))  \cap L^{\infty}(\mathbb{R}^n\times (0,\infty))$, we have
  \begin{equation*}
    \begin{aligned}
      &\int_{0}^{1} {\int_{\mathbb{R}^n}^{} {|u(y, \tau)|^{1+\frac{2}{n}}} \,dy }d\tau \le \int_{0}^{1} {\left\lVert u(\tau) \right\rVert _{L^\infty }^{\frac{2}{n}} \left\lVert u(\tau) \right\rVert _{L^1}}d\tau \le C. 
    \end{aligned}
   \end{equation*}
Therefore, we can prove \eqref{prop3-2} more easily than \eqref{prop3-1}. 
As a result, we are able to see that the asymptotic formula \eqref{w-asymp} is true. This completes the proof. 
\end{proof}

\begin{proof}[\rm{\bf{Proof of Theorem~\ref{duhameleval}}}]
Combining Propositions~\ref{prop1}, \ref{prop2} and \eqref{D-split}, we can immediately see that the desired asymptotic formula \eqref{D1-asymp} is true. This completes the proof. 
\end{proof}

Finally, in the rest of this subsection, let us introduce the asymptotic profile for $D_{2}(x, t)$ in \eqref{IE} when $n\ge2$. 
Actually, it has already been known by the previous works. More precisely, the following formula is given by Duro--Carpio {\rm \cite{D-C}} (for the proof, see Proposition~1 in {\rm \cite{D-C}}). 
\begin{thm}[\cite{D-C}]\label{thm-new1}
Let $n\ge2$ and $ q=1+\frac{2}{n} $. Suppose $ u_0 \in L^{1}( \mathbb{R}^n; 1+|x| )\cap L^{\infty }( \mathbb{R}^n)$. 
Then, for any $ p \in [1, \infty ] $, the following asymptotic formula holds: 
   \begin{equation}\label{w-asymp-2}
    \lim_{t \to \infty} t^{\frac{n}{2} \left( 1-\frac{1}{p} \right)+\frac{1}{2}} \left\lVert D_{2}(t) - \mathcal{N}\cdot \nabla G(t) \right\rVert _{L^p} = 0,
   \end{equation}
   where $ D_{2}(x,t) $, $ G(x,t) $ and $ \mathcal{N} $ are defined by \eqref{IE}, \eqref{heat} and \eqref{DEF-mathN}, respectively.
\end{thm}
\begin{rem}\label{remark1}
In order to prove this proposition, the assumption $n\ge2$ is necessary. 
Actually, if $n=1$, we cannot guarantee $\left|\mathcal{N}\right|<\infty$. Now, let us consider the following integral: 
\begin{align}
\int_{0}^{t}\int_{\R^{n}}\left|b(y)\nabla u(y, \tau)\right|dyd\tau
&\le \int_{0}^{1}\left\|b\right\|_{L^{\infty}}\left\|\nabla u(\tau)\right\|_{L^{1}}d\tau+\int_{1}^{t}\left\|b\right\|_{L^{1}}\left\|\nabla u(\tau)\right\|_{L^{\infty}}d\tau \nonumber \\
&\le C\left\|b\right\|_{L^{\infty}}\int_{0}^{1}\tau^{-\frac{1}{2}}d\tau+C\left\|b\right\|_{L^{1}}\int_{1}^{t}\tau^{-\frac{n}{2}-\frac{1}{2}}d\tau, \ \ t>1. \label{N-finite}
\end{align}
Here, we used \eqref{CDdecay-2} and the following fact (see, Lemma~\ref{lem.nu-est-Lp} below or Lemma~1 in {\rm \cite{D-C}}):
\begin{equation*}
\left\|\nabla u(t)\right\|_{L^{1}}\le Ct^{-\frac{1}{2}}, \ \ t>0. 
\end{equation*}
Therefore, if $n\ge2$, it follows from \eqref{N-finite} that 
\[\left|N\right|\le C\int_{0}^{\infty}\int_{\R^{n}}\left|b(y)\nabla u(y, \tau)\right|dyd\tau<\infty.\] 
On the other hand, the right hand side of \eqref{N-finite} glows with the order $O(\log t)$ when $n=1$. 
\end{rem}

\begin{proof}[\rm{\bf{End of the proof of Theorem~\ref{mainresult-1}}}]
Applying Proposition~\ref{linereval}, Theorems \ref{duhameleval} and \ref{thm-new1} for \eqref{IE}, 
we can conclude that the desired asymptotic formula \eqref{main-1} holds. This completes the proof. 
\end{proof}

\subsection{Proof of Theorem~\ref{mainresult-2}}

In this subsection, we would like to prove Theorem~\ref{mainresult-2}. 
In what follows, let us consider the case of $n=1$ and $q=3$. 
By virtue of Theorem~\ref{duhameleval}, we have already obtained the asymptotic formula for $D_{1}(x, t)$ even if $n=1$. 
Therefore, it is sufficient to derive the leading term of $D_{2}(x, t)$ when $n=1$. In \cite{D-C}, Duro--Carpio derived the asymptotic profile for $D_{2}(x, t)$. 
Here, we shall derive the second asymptotic profile for $D_{2}(x, t)$. 
Indeed, we have succeeded to show the following asymptotic formula: 
\noindent
\begin{thm}\label{thm-D2}
Let $n=1$ and $ q=3$. Suppose $ u_0 \in L^{1}( \mathbb{R}; 1+|x| )\cap L^{\infty }( \mathbb{R} )$. 
Then, for any $ p \in [1, \infty ] $, the following asymptotic formula holds:
  \begin{equation}\label{D2-asymp}
   \lim_{t\to \infty}t^{\frac{1}{2}\left(1-\frac{1}{p}\right)+\frac{1}{2}} \left\lVert D_{2}(t) - K(t)(\log t)\p_{x}G(t) - \Phi (t) \right\rVert _{L^p} =0,
    \end{equation}
where $D_{2}(x, t)$ is defined by \eqref{IE}, while $G(x, t)$, $\Phi(x, t)$ and $K(t)$ are defined by \eqref{heat}, \eqref{DEF-Phi} and \eqref{alpha_n}, respectively. 
\end{thm}

In order to prove this theorem, let us decompose $D_{2}(x, t)$ as follows: 
\begin{align}
D_{2}(x, t)
&=\int_{0}^{1}\left(\p_{x}G(t-\tau)*b\p_{x}u(\tau)\right)(x)d\tau+\int_{1}^{t}\left(\p_{x}G(t-\tau)*b\p_{x}u(\tau)\right)(x)d\tau \nonumber \\
&=:\eta(x, t)+h(x, t). \label{DEF-h}
\end{align}
Moreover, we shall introduce the following function $\mathcal{K}(t)$ and constant $\mathcal{L}$: 
\begin{align}
&\mathcal{K}(t):=\frac{1}{\log t}\int_{1}^{t}\int_{\R}b(y)\p_{x}u(y, \tau)dyd\tau, \ \  \mathcal{L} \coloneqq \int_{0}^{1}\int_{\R}b(y)\p_{x}u(y, \tau)dyd\tau. \label{DEF-mathL}
\end{align}
Then, for $\mathcal{K}(t)$, $\mathcal{L}$ and $K(t)$ in \eqref{alpha_n}, we note that the following relation holds: 
\begin{equation}\label{K-split}
K(t)(\log t)
=\mathcal{K}(t)(\log t)+\mathcal{L}.
\end{equation}

Since $|\mathcal{L}|<\infty$ from \eqref{N-finite}, for the first term $\eta(x, t)$ in the right hand side of \eqref{DEF-h}, we can easily show the following asymptotic formula. 
We omit the proof of this proposition because it can be given in completely the same way to prove Proposition~\ref{prop2}. 
\begin{prop}\label{prop-new2}
Let $n=1$ and $ q=3 $. Suppose $ u_0 \in L^{1}( \mathbb{R}; 1+|x| )\cap L^{\infty }( \mathbb{R} )$. 
Then, for any $ p \in [1, \infty ] $, the following asymptotic formula holds: 
   \begin{equation}\label{eta-asymp}
    \lim_{t \to \infty} t^{\frac{1}{2} \left( 1-\frac{1}{p} \right)+\frac{1}{2}} \left\lVert \eta(t) - \mathcal{L}\p_{x}G(t) \right\rVert _{L^p} = 0,
   \end{equation}
   where $ \eta(x, t) $, $ G(x,t) $ and $ \mathcal{L} $ are defined by \eqref{DEF-h}, \eqref{heat} and \eqref{DEF-mathL}, respectively.
\end{prop}

By virtue of the above proposition, it is sufficient to consider the asymptotic behavior of the second term $h(x, t)$ in the right hand side of \eqref{DEF-h}. 
In order to do that, we shall prepare the decay estimate for $\nabla u(x, t)$. 
Actually, such an estimate has already been introduced in \eqref{CDdecay-2} for $t>1$. 
However, our results require that the estimate of \eqref{CDdecay-2} holds not only for $t>1$ but also for $t>0$. 
For deriving that estimate, let us apply the following Gronwall's type lemma (for the proof, see e.g. \cite{Z1}): 
\begin{lemma}[\cite{Z1}]\label{Gronwall}
Let $T>0$, $A, B>0$ and $\alpha\in [0, 1)$. Suppose that $\varphi \in C([0, T])$ be a non-negative function satisfying the following inequality:  
\[
\varphi(t) \le At^{-\alpha}+B\int_{0}^{t}(t-\tau)^{-\frac{1}{2}}\varphi(\tau) d\tau, \ \ t\in (0, T].
\]
Then, there exists a constant $C>0$ depending only $T$, $B$ and $\alpha$ such that 
\[
\varphi(t)\le CAt^{-\alpha}, \ \ t\in (0, T].
\]
\end{lemma}

Now, we would like to prove the $L^{p}$-decay estimate of $\nabla u(x, t)$ for any $p\in [1, \infty]$ and $t>0$. 
Actually, the following estimate \eqref{dxu-infty-Lp} has already been obtained for $p=1$, by Duro--Carpio \cite{D-C}. 
Here, we generalize it and give the proof by slightly modifying the technique used in \cite{D-C, DZ99, Z1}. 
\begin{lemma}\label{lem.nu-est-Lp}
Let $u(x, t)$ be the solution to \eqref{CD} with $q>1+\frac{1}{n}$ and $u_{0}\in L^{1}(\R^{n})\cap L^{n(q-1)}(\R^{n})$. 
Then, for any $p\in [1, \infty]$, there exists a constant $C>0$ depending on $\|u_{0}\|_{L^{1}}$ and $\|u_{0}\|_{L^{n(q-1)}}$ such that  
\begin{equation}\label{dxu-infty-Lp}
\left\|\nabla u(t)\right\|_{L^{p}}\le Ct^{-\frac{n}{2}\left(1-\frac{1}{p}\right)-\frac{1}{2}}, \ \ t>0. 
\end{equation}
\end{lemma}
\begin{proof}
First, let us introduce the following rescaled functions: 
\begin{equation}\label{u-lambda}
u_{\lambda}(x, t):=\lambda^{n}u\left(\lambda x, \lambda^{2}t\right), \ \ \lambda>0. 
\end{equation}
Then, since $u(x, t)$ is the solution to \eqref{CD}, these functions satisfy  
\begin{equation}\label{CD-rescal}
  \begin{aligned}
    &\partial_t u_{\lambda} -  \mathrm{div}\left(a(\lambda x)\nabla u_{\lambda}\right) = \lambda^{n(1-q)+1}d\cdot\nabla  \left(\left\lvert u_{\lambda}\right\rvert ^{q-1}u_{\lambda}\right),\ \  x\in\mathbb{R}^n, \ t>0,\\
    &u_{\lambda}(x,0)=u_{\lambda, 0}(x)=\lambda^{n}u_{0}(\lambda x),\ \  x\in\mathbb{R}^n. 
\end{aligned}
\end{equation}
Therefore, for any fixed $s>0$, we can derive the following integral equations: 
\begin{equation*}
u_{\lambda}(t+s)=T_{\lambda}(t)\left[u_{\lambda}(s)\right]+\lambda^{n(1-q)+1}\int_{0}^{t}T_{\lambda}(t-\tau)d\cdot \left[\nabla\left(F\left(u_{\lambda}(\tau+s)\right)\right)\right]d\tau, 
\end{equation*}
where $F(u_{\lambda}):=\left\lvert u_{\lambda}\right\rvert ^{q-1}u_{\lambda}$ and $T_{\lambda}(t)$ is the semigroup generated by the linear diffusion equation with coefficient $a_{\lambda}(x):=a(\lambda x)$. 

Now, taking the gradient in the above integral equations, we have 
\begin{equation}\label{int-nu(t+s)}
\nabla u_{\lambda}(t+s)=\nabla \left(T_{\lambda}(t)\left[u_{\lambda}(s)\right]\right)+\lambda^{n(1-q)+1}\int_{0}^{t}\nabla \left(T_{\lambda}(t-\tau)d\cdot \left[\nabla\left(F\left(u_{\lambda}(\tau+s)\right)\right)\right]\right)d\tau. 
\end{equation}
Here, we recall that $T_{\lambda}(t)$ satisfies the following estimate: 
\begin{equation}\label{est-semigroup}
\left\|\nabla \left(T_{\lambda}(t)\left[v\right]\right)\right\|_{L^{p}}\le C\left\|v\right\|_{L^{r}}t^{-\frac{n}{2}\left(\frac{1}{r}-\frac{1}{p}\right)-\frac{1}{2}}, \ \ t>0, \ \lambda>0, \ 1\le r\le p\le \infty
\end{equation}
with some positive constant $C>0$ independent of $\lambda$. Moreover, we note that the solution $u(x, t)$ also satisfies
\begin{equation*}
\left\|u(t)\right\|_{L^{p}}\le C\left\|u_{0}\right\|_{L^{r}}t^{-\frac{n}{2}\left(\frac{1}{r}-\frac{1}{p}\right)}, \ \ t>0, \ 1\le r\le p\le \infty. 
\end{equation*}
For their two proofs, see Proposition~1 in \cite{DZ99}. 
Furthermore, since $u_{\lambda}(x, t)$ is the solution to \eqref{CD-rescal}, analogously as the above estimate, we can see that $u_{\lambda}(x, t)$ also satisfies 
\begin{equation*}
\left\|u_{\lambda}(t)\right\|_{L^{p}}\le C\left\|u_{\lambda, 0}\right\|_{L^{r}}t^{-\frac{n}{2}\left(\frac{1}{r}-\frac{1}{p}\right)}, \ \ t>0, \ \lambda>0, \ 1\le r\le p\le \infty. 
\end{equation*}
Therefore, it follows from the above estimate that 
\[
\left\|u_{\lambda}(t)\right\|_{L^{\infty}}\le Ct^{-\frac{n}{2r}}\lambda^{n\left(1-\frac{1}{r}\right)}\left\|u_{0}\right\|_{L^{r}}, \ \ t>0, \ \lambda>0, \ 1\le r\le \infty
\]
with some positive constant $C>0$ independent of $\lambda$. From the above estimate, we obtain 
\begin{equation}\label{Lp-Lq-est-u}
\left\|u_{\lambda}(\tau+s)\right\|_{L^{\infty}}\le Cs^{-\frac{n}{2r}}\lambda^{n\left(1-\frac{1}{r}\right)}\left\|u_{0}\right\|_{L^{r}}, \ \ \tau\ge0, \ s>0, \ \lambda>0, \ 1\le r\le \infty. 
\end{equation}
Thus, taking $r=n(q-1)$, it follows from \eqref{int-nu(t+s)}, \eqref{est-semigroup} and \eqref{Lp-Lq-est-u} that 
\begin{align}
\left\|\nabla u_{\lambda}(t+s)\right\|_{L^{p}}
&\le C\left\|u_{0}\right\|_{L^{1}}t^{-\frac{n}{2}\left(1-\frac{1}{p}\right)-\frac{1}{2}} \nonumber \\
&\ \ \ +C|d|q\lambda^{n(1-q)+1}\int_{0}^{t}(t-\tau)^{-\frac{1}{2}}\left\|u_{\lambda}(\tau+s)\right\|_{L^{\infty}}^{q-1}\left\|\nabla u_{\lambda}(\tau+s)\right\|_{L^{p}}d\tau 
\nonumber \\
&\le C\left\|u_{0}\right\|_{L^{1}}t^{-\frac{n}{2}\left(1-\frac{1}{p}\right)-\frac{1}{2}} \nonumber \\
&\ \ \ +C|d|q\left\|u_{0}\right\|_{L^{n(q-1)}}^{q-1}s^{-\frac{1}{2}}\int_{0}^{t}(t-\tau)^{-\frac{1}{2}}\left\|\nabla u_{\lambda}(\tau+s)\right\|_{L^{p}}d\tau, \label{nu(t+s)-est-pre}
\end{align}
for any $p\in [1, \infty]$, $t>0$, $s>0$ and $\lambda>0$. 

Applying Lemma~\ref{Gronwall} for \eqref{nu(t+s)-est-pre} with $p\ge1$ satisfying $\frac{n}{2}\left(1-\frac{1}{p}\right)+\frac{1}{2}<1$, i.e., $p\in \left[1, \frac{n}{n-1}\right)$, we have the following estimate: 
\[
\left\|\nabla u_{\lambda}(t+s)\right\|_{L^{p}}\le C_{s}\|u_{0}\|_{L^{1}}t^{-\frac{n}{2}\left(1-\frac{1}{p}\right)-\frac{1}{2}}, \ \ t\in [0, 1], \ s>0, \ \lambda>0. 
\]
Therefore, taking $t=s=\frac{1}{2}$ in the above estimate, we obtain 
\[
\left\|\nabla u_{\lambda}(1)\right\|_{L^{p}}\le C\|u_{0}\|_{L^{1}}, \ \ \lambda>0, \ p\in \left[1, \frac{n}{n-1}\right). 
\]
Now, substituting $\lambda=\sqrt{t}>0$ in \eqref{u-lambda}, we have $\nabla u_{\sqrt{t}}(x, 1)=t^{\frac{n}{2}+\frac{1}{2}}\nabla u\left(\sqrt{t}x, t\right)$. Hence, by using the above estimate with $\lambda=\sqrt{t}>0$, we eventually arrive at 
\begin{equation}\label{dxu-infty-Lp-pre}
t^{\frac{n}{2}\left(1-\frac{1}{p}\right)+\frac{1}{2}}\left\|\nabla u(t)\right\|_{L^{p}}
=\left\|\nabla u_{\sqrt{t}}(1)\right\|\le C\|u_{0}\|_{L^{1}}, \ \ t>0, \ p\in \left[1, \frac{n}{n-1}\right). 
\end{equation}
This result means that the desired estimate \eqref{dxu-infty-Lp} holds for any $p\in \left[1, \frac{n}{n-1}\right)$. 

In order to prove \eqref{dxu-infty-Lp} for any $p\in [1, \infty]$, we shall use an iterative argument. 
In what follows, let us consider the case of $p\ge \frac{n}{n-1}$. Now, it follows from \eqref{dxu-infty-Lp-pre} that 
\begin{equation}\label{dxu-Lp-pre2}
\left\|\nabla u_{\lambda}(\tau+s)\right\|_{L^{p}}\le C_{s}, \ \ \tau\ge0, \ s>0, \ \lambda>0, \ p\in \left[1, \frac{n}{n-1}\right). 
\end{equation}
Now, taking the parameter $r\in \left[1, \frac{n}{n-1}\right)$ and choosing $p$ such that $\frac{n}{2}\left(\frac{1}{r}-\frac{1}{p}\right)+\frac{1}{2}<1$. In this case, we note that $p$ satisfies $p\in \left[\frac{n}{n-1}, \frac{n}{n-2}\right)$. 
Under this situation, in the same way to get \eqref{nu(t+s)-est-pre}, it follows from \eqref{int-nu(t+s)}, \eqref{est-semigroup}, \eqref{Lp-Lq-est-u} and \eqref{dxu-Lp-pre2} that 
\begin{align}
\left\|\nabla u_{\lambda}(t+s)\right\|_{L^{p}}
&\le C\left\|u_{0}\right\|_{L^{1}}t^{-\frac{n}{2}\left(1-\frac{1}{p}\right)-\frac{1}{2}} \nonumber \\
&\ \ \ +C|d|q\lambda^{n(1-q)+1}\int_{0}^{t}(t-\tau)^{-\frac{n}{2}\left(\frac{1}{r}-\frac{1}{p}\right)-\frac{1}{2}}\left\|u_{\lambda}(\tau+s)\right\|_{L^{\infty}}^{q-1}\left\|\nabla u_{\lambda}(\tau+s)\right\|_{L^{r}}d\tau 
\nonumber \\
&\le C\left\|u_{0}\right\|_{L^{1}}t^{-\frac{n}{2}\left(1-\frac{1}{p}\right)-\frac{1}{2}}
+C_{s}|d|q\left\|u_{0}\right\|_{L^{n(q-1)}}^{q-1}\int_{0}^{t}(t-\tau)^{-\frac{n}{2}\left(\frac{1}{r}-\frac{1}{p}\right)-\frac{1}{2}}d\tau \nonumber \\
&\le C\left\{t^{-\frac{n}{2}\left(1-\frac{1}{p}\right)-\frac{1}{2}}
+t^{-\frac{n}{2}\left(\frac{1}{r}-\frac{1}{p}\right)+\frac{1}{2}}
\right\}, \nonumber 
\end{align}
for any $p\in \left[\frac{n}{n-1}, \frac{n}{n-2}\right)$, $t>0$, $s>0$ and $\lambda>0$ with some positive constant $C>0$ depending on $\|u_{0}\|_{L^{1}}$ and $\|u_{0}\|_{L^{n(q-1)}}$. Therefore, taking $t=s=\frac{1}{2}$ in the above estimate, we arrive at 
\[
\left\|\nabla u_{\lambda}(1)\right\|_{L^{p}}\le C, \ \ \lambda>0, \ p\in \left[\frac{n}{n-1}, \frac{n}{n-2}\right). 
\]
In the same way to get \eqref{dxu-infty-Lp-pre}, we eventually see that \eqref{dxu-infty-Lp} holds for any $p\in \left[\frac{n}{n-1}, \frac{n}{n-2}\right)$. 
Finally, iterating this argument, it can be concluded that the desired estimate \eqref{dxu-infty-Lp} is true for any $p\in [1, \infty]$. This completes the proof. 
\end{proof}

Finally, let us consider the asymptotic behavior of $h(x, t)$ in \eqref{DEF-h}. 
We have slightly modified the proof of Proposition~\ref{prop1} and succeeded to derive the following asymptotic formula: 
\begin{prop}\label{prop-new-h}
Let $n=1$ and $ q=3$. Suppose $ u_0 \in L^{1}( \mathbb{R}; 1+|x| )\cap L^{\infty }( \mathbb{R} )$. 
Then, for any $ p \in [1, \infty ] $, the following asymptotic formula holds:
  \begin{equation}\label{h-asymp}
    \left\lVert h(t) - \mathcal{K}(t)(\log t)\p_{x}G(t) - \Phi (t) \right\rVert _{L^p} 
    =\begin{cases}
        O\left(t^{-\frac{1}{2}\left(1-\frac{1}{p}\right)-1}\right), &\displaystyle \delta \ge1,\\[4mm]
        o\left(t^{-\frac{1}{2}\left(1-\frac{1}{p}\right)-\frac{1}{2}}\right), &\displaystyle 0<\delta<1 
    \end{cases}  
    \end{equation}
    as $t\to \infty$, where $ h(x,t) $, $ G(x,t) $, $ \Phi (x,t) $ and $\mathcal{K}(t)$ are defined by \eqref{DEF-h}, \eqref{heat}, \eqref{DEF-Phi} and \eqref{DEF-mathL}, respectively. 
    In addition, $\delta>0$ is the constant appeared in \eqref{C-b}. 
\end{prop}
\begin{proof}
  First, we set $H(x, t)$ as follows: 
  \[
  H(x, t):=\mathcal{K}(t)(\log t)\p_{x}G(x, t)=\left(\int_{1}^{t}\int_{\R}b(y)\p_{x}u(y, \tau)dyd\tau\right)\p_{x}G(x, t). 
  \]
  Then, in the same way to get \eqref{CP-V}, we can see that $H(x, t)$ satisfies 
  \begin{equation*}
    \begin{aligned}
      &\partial _t H - \p_{x}^{2}H = \left(\int_{\R}b(y)\p_{x}u(y, t)dy\right)\p_{x}G(x, t), \ \ x\in \R, \ t>1, \\
      &H(x, 1) = 0, \ \ x\in \R.
    \end{aligned}
  \end{equation*}
Applying the Duhamel principle, $ H(x,t) $ can be transformed as follows: 
    \begin{align*}
      H(x, t) 
      = \int_{1}^{t}\int_{\R}\p_{x}G(x-y, t-\tau)\left(\int_{\R}b(z)\p_{x}u(z, \tau)dz\right)G(y, \tau)dyd\tau. 
    \end{align*}
  Therefore, it follows from \eqref{DEF-h}, \eqref{DEF-phi} and the above expression that 
    \begin{align}
      &h(x, t) - H(x, t) \nonumber \\
      &=\int_{1}^{t}\int_{\R}\p_{x}G(x-y, t-\tau)
      \left\{b(y)\p_{x}u(y, \tau)-\left(\int_{\R}b(z)\p_{x}u(z, \tau)dz\right)G(y, \tau)\right\}dyd\tau \nonumber \\
      &=\p_{x}\int_{1}^{t}\int_{\R}G(x-y, t-\tau)\phi(y, \tau)dyd\tau=:\p_{x}E(x, t). \label{h-H}
         \end{align}
     
  Next, let us transform $ E(x, t) $ in the above equation \eqref{h-H}, by using the scaling argument again. 
  Now, recalling \eqref{scal-2}, it follows from the change of variables that 
  \begin{align}
    E(x, t) &=t^{-\frac{1}{2}}\int_{1}^{t}\int_{\R}G\left(\frac{x-y}{\sqrt{t}}, 1-\frac{\tau}{t}\right)\phi(y, \tau)dyd\tau \nonumber \\
        &= t \int_{\frac{1}{t}}^{1} \int_{\mathbb{R}}G\left(\frac{x}{\sqrt{t}} - z, 1 - s \right) \phi\left(\sqrt{t}z, ts\right)dzds\nonumber \\
    &= t \int_{\frac{1}{t}}^{1} {\left( G(1-s) \ast \phi\left(\sqrt{t}\,\cdot, ts\right) \right) \left( \frac{x}{\sqrt{t}} \right)}ds. \label{E-henkei}
  \end{align}
  Here, similarly as \eqref{Psi-t}, from the definition of $\Phi(x, t)$ in \eqref{DEF-Phi}, we get 
  \begin{align*}
    \Phi(x, t)
    &=t^{-1}\Phi _\ast \left( \frac{x}{\sqrt{t}}, t \right)
    =t^{\frac{1}{2}}\p_{x}\left( \int_{0}^{1} {\left( G(1-s) \ast \phi\left(\sqrt{t}\,\cdot, ts\right)\right) \left( x \right)} \,ds \right)\biggl|_{x=\frac{x}{\sqrt{t}}}  \\
&    =t \,\p_{x}\left(\int_{0}^{1} {\left( G(1-s) \ast \phi\left(\sqrt{t}\,\cdot, ts\right)\right) \left( \frac{x}{\sqrt{t}} \right)}ds\right).
  \end{align*}
  Hence, for any $p\in [1, \infty]$, we have from \eqref{h-H} and \eqref{E-henkei} that 
  \begin{align}
    \left\lVert h(t) - H(t) - \Phi (t) \right\rVert _{L^p} 
    &= t \left\lVert \p_{x}\left(\int_{0}^{\frac{1}{t}} {\left( G(1-s) \ast \phi\left(\sqrt{t}\,\cdot, ts\right)\right) \left( \frac{x }{\sqrt{t}} \right)}ds\right)\right\rVert _{L^{p}_{x}} \nonumber \\
    &= t^{\frac{1}{2}+\frac{1}{2p}} \left\lVert \p_{x}\int_{0}^{\frac{1}{t}} {\left(G(1-s) \ast \phi\left(\sqrt{t}\,\cdot, ts\right)\right)}(x)ds\right\rVert _{L^{p}_{x}}. \label{mainpart-2}
  \end{align}
  
  In what follows, let us evaluate the right hand side of the above \eqref{mainpart-2}. 
  First, we shall consider the case of $\delta\ge1$. In this case, noticing that $|x|b\in L^{\infty}(\R)$ holds from \eqref{C-b}. 
  Recalling the definition of $\phi(y, \tau)$ in \eqref{DEF-phi} and noticing $\int_{\R}\phi(y, \tau)dy=0$ for all $\tau>0$, analogously as \eqref{dag-est}, it follows from \eqref{Gdecay} and \eqref{dxu-infty-Lp} that 
  \begin{align} 
        &\left\| \p_{x}\int_0^{\frac{1}{t}} \left(G(1-s) \ast \phi\left(\sqrt{t}\,\cdot, ts\right)\right)(x)ds \right\|_{L^{p}_{x}} \nonumber \\
        &=\left\| \int_{0}^{\frac{1}{t}} \int_{\R} \left( \int_{0}^{1} \p_{x}^{2} G(x - \theta y, 1 - s) d\theta \right) y\phi\left(\sqrt{t}y, ts\right)dy ds \right\|_{L^{p}_{x}} \nonumber \\
        &\le t^{-\frac{1}{2}}\int_{0}^{\frac{1}{t}} \int_{0}^{1} \left\| \int_{\R} \p_{x}^{2}\left( \frac{1}{\theta}G\left(\frac{x}{\theta} - y, \frac{1 - s}{\theta^{2}}\right) \right) \sqrt{t}y\phi\left(\sqrt{t}y, ts\right)dy \right\|_{L^{p}_{x}}d\theta ds \nonumber \\
        &\le Ct^{-1} \int_{0}^{\frac{1}{t}} \left(\int_{0}^{1} \theta^{-1} \theta^{-2}\theta^{\frac{1}{p}}\theta^{1-\frac{1}{p}+2}d\theta \right)(1-s)^{-\frac{1}{2}\left(1-\frac{1}{p}\right)-1}
        \left\||\cdot|\phi\left(ts\right)\right\|_{L^{1}}ds \nonumber \\
        &\le Ct^{-1}\int_{0}^{\frac{1}{t}}(1-s)^{-\frac{1}{2}\left(1-\frac{1}{p}\right)-1}\left(\left\||\cdot|b\right\|_{L^{\infty}}\|\p_{x}u(ts)\|_{L^{1}}+\|\p_{x}u(ts)\|_{L^{\infty}}\|b\|_{L^{1}}\left\||\cdot|G(ts)\right\|_{L^{1}}\right)ds \nonumber \\
        &\le C\left(\||\cdot|b\|_{L^{\infty}}+\|b\|_{L^{1}}\right)t^{-1}\int_{0}^{\frac{1}{t}}(1-s)^{-\frac{1}{2}\left(1-\frac{1}{p}\right)-1}\left\{(ts)^{-\frac{1}{2}}+(ts)^{-1}(ts)^{\frac{1}{2}}\right\}ds \nonumber \\
        &\le C\left(\||\cdot|b\|_{L^{\infty}}+\|b\|_{L^{1}}\right)t^{-\frac{3}{2}}\int_{0}^{\frac{1}{t}}s^{-\frac{1}{2}}ds 
         \le C\left(\||\cdot|b\|_{L^{\infty}}+\|b\|_{L^{1}}\right)t^{-2}, \ \ t\ge2,  \label{final-est}
    \end{align}
    for any $p \in [1, \infty]$. 
    Summarizing up \eqref{mainpart-2} and \eqref{final-est}, we can eventually conclude that the desired result \eqref{h-asymp} is true for $\delta\ge1$. 
    
    Next, we would like to deal with the case of $0<\delta<1$. In this case, the argument given in \eqref{final-est} needs to be modified, because we cannot guarantee that $|x|b\in L^{\infty}(\R)$ holds. Now, let us take $\varepsilon \in (0, 1)$ and split the integral in \eqref{mainpart-2} as follows:  
    \begin{align}
    &\p_{x}\int_0^{\frac{1}{t}} \left(G(1-s) \ast \phi\left(\sqrt{t}\,\cdot, ts\right)\right)(x)ds \nonumber \\
    &=\int_0^{\frac{1}{t}}\int_{\R}\p_{x}\left\{G(x-y, 1-s)-G(x, 1-s)\right\}\phi\left(\sqrt{t}y, ts\right)dy ds\nonumber  \\
    &=\int_0^{\frac{1}{t}}\int_{|y|\le \sqrt{s} \varepsilon^{-\frac{1}{\delta}} }\p_{x}\left\{G(x-y, 1-s)-G(x, 1-s)\right\}b\left(\sqrt{t}y\right)\p_{x}u\left(\sqrt{t}y, ts\right)dy ds \nonumber \\
    &\ \ \ +\int_0^{\frac{1}{t}}\int_{|y|\ge \sqrt{s} \varepsilon^{-\frac{1}{\delta}} }\p_{x}\left\{G(x-y, 1-s)-G(x, 1-s)\right\}b\left(\sqrt{t}y\right)\p_{x}u\left(\sqrt{t}y, ts\right)dyds\nonumber  \\
    &\ \ \ -\int_0^{\frac{1}{t}}\int_{\R}\p_{x}\left\{G(x-y, 1-s)-G(x, 1-s)\right\}\left(\int_{\R}b(z)\p_{x}u(z, ts)dz\right)G\left(\sqrt{t}y, ts\right)dyds \nonumber \\
    &=\int_0^{\frac{1}{t}}\int_{|y|\le \sqrt{s} \varepsilon^{-\frac{1}{\delta}} }\left(\int_{0}^{1}\p_{x}^{2}G(x-\theta y, 1-s)d\theta\right)(-y)b\left(\sqrt{t}y\right)\p_{x}u\left(\sqrt{t}y, ts\right)dy ds \nonumber \\
        &\ \ \ +\int_0^{\frac{1}{t}}\int_{|y|\ge \sqrt{s} \varepsilon^{-\frac{1}{\delta}} }\p_{x}\left\{G(x-y, 1-s)-G(x, 1-s)\right\}b\left(\sqrt{t}y\right)\p_{x}u\left(\sqrt{t}y, ts\right)dyds \nonumber \\
&\ \ \ -\int_0^{\frac{1}{t}}\int_{\R}\left(\int_{0}^{1}\p_{x}^{2}G(x-\theta y, 1-s)d\theta\right)(-y)\left(\int_{\R}b(z)\p_{x}u(z, ts)dz\right)G\left(\sqrt{t}y, ts\right)dy ds \nonumber  \\
&=:I_{1}(x, t)+I_{2}(x, t)+I_{3}(x, t),  \label{I1+I2+I3}
    \end{align}
    where we used the definition of $\phi(y, \tau)$ in \eqref{DEF-phi} and $\int_{\R}\phi(y, \tau)dy=0$ again. In what follows, we shall evaluate $I_{1}(x, t)$, $I_{2}(x, t)$ and $I_{3}(x, t)$. 
    First, for $I_{1}(x, t)$, using the similar argument to \eqref{final-est} and a direct calculation, it follows from \eqref{C-b} and \eqref{dxu-infty-Lp} that 
    \begin{align}
    \left\|I_{1}(t)\right\|_{L^{p}}
    &\le Ct^{-1}\int_{0}^{\frac{1}{t}}(1-s)^{-\frac{1}{2}\left(1-\frac{1}{p}\right)-1}\left(\int_{|w|\le \sqrt{ts} \varepsilon^{-\frac{1}{\delta}} } |w||b(w)|\left|\p_{x}u(w, ts)\right|dw\right)ds \nonumber \\
    &\le Ct^{-1}\int_{0}^{\frac{1}{t}}\sqrt{ts}\varepsilon^{-\frac{1}{\delta}} \|\p_{x}u(ts)\|_{L^{\infty}}\left(\int_{|w|\le \sqrt{ts} \varepsilon^{-\frac{1}{\delta}} } \left(1+|w|^{2}\right)^{-\frac{\delta}{2}}dw\right) \nonumber  \\
    &\le Ct^{-1}\int_{0}^{\frac{1}{t}}\sqrt{ts}\varepsilon^{-\frac{1}{\delta}} (ts)^{-1}\left(\int_{0}^{\sqrt{ts} \varepsilon^{-\frac{1}{\delta}} } r^{-\delta}dr\right)  
    \le C_{\varepsilon, \delta}\,t^{-1}\int_{0}^{\frac{1}{t}}(ts)^{-\frac{1}{2}}(ts)^{\frac{1-\delta}{2}}ds \nonumber \\
&\le C_{\varepsilon, \delta}\,t^{-1-\frac{\delta}{2}}\int_{0}^{\frac{1}{t}}s^{-\frac{\delta}{2}}ds 
\le C_{\varepsilon, \delta}\,t^{-2}, \ \ t\ge2. \label{I1-est}
    \end{align}
Moreover, by using \eqref{C-b} and \eqref{dxu-infty-Lp} again, we can evaluate $I_{2}(x, t)$ as follows: 
\begin{align}
    \left\|I_{2}(t)\right\|_{L^{p}} 
    &\le \int_{0}^{\frac{1}{t}}\int_{|y|\ge \sqrt{s} \varepsilon^{-\frac{1}{\delta}} }\left(\left\|\p_{x}G(\cdot -y, 1-s)\right\|_{L^{p}}+\left\|\p_{x}G(1-s)\right\|_{L^{p}}\right) \nonumber \\
    &\ \ \ \ \times\left|b\left(\sqrt{t}y\right)\right|\left|\p_{x}u\left(\sqrt{t}y, ts\right)\right|dyds \nonumber \\
    &\le Ct^{-\frac{1}{2}}\int_{0}^{\frac{1}{t}}(1-s)^{-\frac{1}{2}\left(1-\frac{1}{p}\right)-\frac{1}{2}}\left(\int_{|w|\ge \sqrt{ts} \varepsilon^{-\frac{1}{\delta}} }\left|b(w)\right|\left|\p_{x}u(w, ts)\right|dw\right)ds \nonumber \\
    &\le Ct^{-\frac{1}{2}}\int_{0}^{\frac{1}{t}}\int_{|w|\ge \sqrt{ts} \varepsilon^{-\frac{1}{\delta}} }\left(1+|w|^{2}\right)^{-\frac{\delta}{2}}\left|\p_{x}u(w, ts)\right|dwds  \nonumber \\
    &\le Ct^{-\frac{1}{2}}\int_{0}^{\frac{1}{t}}\left(\sup_{|w|\ge \sqrt{ts} \varepsilon^{-\frac{1}{\delta}}}|w|^{-\delta}\right) \|\p_{x}u(ts)\|_{L^{1}} ds 
    \le C\varepsilon t^{-\frac{1}{2}}\int_{0}^{\frac{1}{t}} (ts)^{-\frac{\delta}{2}}(ts)^{-\frac{1}{2}}ds \nonumber \\
    &\le C\varepsilon t^{-1-\frac{\delta}{2}}\int_{0}^{\frac{1}{t}}s^{-\frac{1}{2}-\frac{\delta}{2}}ds
    \le C_{\delta}\varepsilon t^{-\frac{3}{2}}, \ \ t\ge2. \label{I2-est}
\end{align}
Finally, in the same way to get \eqref{final-est}, we note that $I_{3}(x, t)$ can be evaluated as follows: 
\begin{align}
    \left\|I_{3}(t)\right\|_{L^{p}} 
    \le C\|b\|_{L^{1}}t^{-2}, \ \ t\ge2. \label{I3-est}
\end{align}

Summarizing up \eqref{I1-est}, \eqref{I2-est} and \eqref{I3-est}, we eventually obtain the following estimate: 
\begin{equation}\label{I1-I3-comb}
\left\| \p_{x}\int_0^{\frac{1}{t}} \left(G(1-s) \ast \phi\left(\sqrt{t}\,\cdot, ts\right)\right)(x)ds \right\|_{L^{p}_{x}} 
\le C_{\varepsilon, \delta}\,t^{-2}+C_{\delta}\varepsilon t^{-\frac{3}{2}}+C\|b\|_{L^{1}}t^{-2}, \ \ t\ge2,  
\end{equation}
    for any $p \in [1, \infty]$. Combining \eqref{mainpart-2} and \eqref{I1-I3-comb}, we have
\[
 \limsup_{t\to \infty}t^{\frac{1}{2}\left(1-\frac{1}{p}\right)+\frac{1}{2}}\left\lVert h(t) - \mathcal{K}(t)(\log t)\p_{x}G(t) - \Phi (t) \right\rVert _{L^p}  \le C_{\delta}\varepsilon. 
\]
Therefore, we are able to see that the desired formula \eqref{h-asymp} holds for $0<\delta<1$ because $\varepsilon>0$ can be chosen arbitrarily small. 
This completes the proof.  
\end{proof}

\begin{proof}[\rm{\bf{Proof of Theorem~\ref{thm-D2}}}]
Combining Propositions~\ref{prop-new2}, \ref{prop-new-h}, \eqref{DEF-h} and \eqref{K-split}, we can immediately see that the desired asymptotic formula \eqref{D2-asymp} is true. This completes the proof. 
\end{proof}

\begin{proof}[\rm{\bf{End of the proof of Theorem~\ref{mainresult-2}}}]
Applying Proposition~\ref{linereval}, Theorems \ref{duhameleval} and \ref{thm-D2} for \eqref{IE}, 
we can conclude that the desired asymptotic formula \eqref{main-2} holds. This completes the proof. 
\end{proof}

\section*{Acknowledgment}

This study is supported by Grant-in-Aid for Young Scientists Research No.22K13939, Japan Society for the Promotion of Science. 
The authors would like to express their appreciation to Professor Ryosuke Nakasato and Mr.\,Ryunosuke Kusaba for their useful comments and stimulating discussions.
The authors also would like to thank the anonymous referee for their helpful and valuable comments on the paper.



\vskip5pt
\begin{flushleft}
Ikki Fukuda\\
Division of Mathematics and Physics, \\
Faculty of Engineering, Shinshu University\\
4-17-1, Wakasato, Nagano, 380-8553, JAPAN\\
E-mail: i\_fukuda@shinshu-u.ac.jp

\vskip5pt
Shinya Sato\\
Department of Engineering, \\
Graduate School of Science and Technology, Shinshu University\\
4-17-1, Wakasato, Nagano, 380-8553, JAPAN\\
E-mail: sato.shinya@qri.jp
\end{flushleft}

\end{document}